\documentclass[11pt,leqno]{amsart}

\usepackage{graphicx}
\graphicspath{ {./images/} }

\usepackage[dvipsnames]{xcolor}
\definecolor{b}{HTML}{4472c4}
\definecolor{o}{HTML}{ED7D31}
\definecolor{g}{HTML}{70ad47}

\usepackage{fullpage}
\usepackage{amsmath,amsthm,amssymb,latexsym,color}
\usepackage[hidelinks,unicode,linktocpage=true]{hyperref}
\usepackage{lipsum}
\usepackage{enumitem,bbm}
\usepackage{hyperref}
\hypersetup{
    colorlinks=true,
    linkcolor=blue,
    filecolor=blue,
    citecolor=blue,
    urlcolor=cyan}

\usepackage{tikz}
\usepackage{subfigure}

\usepackage{multirow}

\newtheorem{theorem}{Theorem}[section]

\newtheorem{lemma}[theorem]{Lemma}
\newtheorem{corollary}[theorem]{Corollary}
\newtheorem{definition}[theorem]{Definition}
\newtheorem{proposition}[theorem]{Proposition}
\newtheorem{fact}[theorem]{Fact}
\newtheorem{claim}[theorem]{Claim}

\newtheorem{problem}[theorem]{Problem}

\theoremstyle{remark}
\newtheorem{remark}[theorem]{Remark}

\newtheorem{convention}[theorem]{Convention}

\theoremstyle{plain}

\newcommand{\Z}{\mathbb{Z}}
\newcommand{\R}{\mathbb{R}}

\newcommand{\Exp}{\mathbb{E}}

\def\Prob{{\mathbb P}}

\newcommand{\supp}{{\rm supp\,}}

\def\dist{{\rm dist}}

\def\dist{{\rm dist}}

\newcommand{\eps}{\epsilon}

\newcommand{\seminorm}[1]{{\left\vert\kern-0.25ex\left\vert\kern-0.25ex\left\vert #1
    \right\vert\kern-0.25ex\right\vert\kern-0.25ex\right\vert}}

\begin{document}

\title[A combinatorial approach to nonlinear spectral gaps]{A combinatorial approach to nonlinear spectral gaps}

\author{Dylan J. Altschuler, Pandelis Dodos, Konstantin Tikhomirov and Konstantinos Tyros}

\address{Department of Mathematical Sciences, Carnegie Mellon University}
\email{daltschu@andrew.cmu.edu}

\address{Department of Mathematics, University of Athens, Panepistimiopolis 157 84, Athens, Greece}
\email{pdodos@math.uoa.gr}

\address{Department of Mathematical Sciences, Carnegie Mellon University}
\email{ktikhomi@andrew.cmu.edu}

\address{Department of Mathematics, University of Athens, Panepistimiopolis 157 84, Athens, Greece}
\email{ktyros@math.uoa.gr}

\thanks{2020 \textit{Mathematics Subject Classification}: 05C12, 05C48, 05C50, 05C80, 30L05, 46B07, 46B09, 46B85.}
\thanks{\textit{Key words}: random regular graphs, Poincar\'{e} inequalities, unconditional bases, cotype.}

%------------------------Abstract-------------------------------%

\begin{abstract}
A seminal open question of Pisier and Mendel--Naor asks whether every degree-regular graph which satisfies the classical discrete Poincar\'e inequality for scalar functions also satisfies an analogous inequality for functions taking values in \textit{any} normed space with non-trivial cotype. Motivated by applications, it is also of great importance to quantify the dependence of the corresponding optimal Poincar\'e constant on the cotype $q$. Works of Odell--Schlumprecht (1994), Ozawa (2004), and Naor (2014) make substantial progress on the former question by providing a positive answer for normed spaces which also have an unconditional basis, in addition to finite cotype. However, little is known in the way of quantitative estimates: the mentioned results imply a bound on the Poincar\'e constant depending super-exponentially on $q$.

We introduce a novel combinatorial framework for proving quantitative nonlinear spectral gap estimates. The centerpiece is a property of regular graphs that we call \textit{long-range expansion}, which holds with high probability for random regular graphs. Our main result is that any regular graph with the long-range expansion property satisfies a discrete Poincar\'{e} inequality for any normed space with an unconditional basis and cotype $q$, with a Poincar\'{e} constant that depends {\it polynomially} on~$q$, which is optimal. As an application, any normed space with an unconditional basis which admits a low-distortion embedding of an $n$-vertex random regular graph must have cotype at least polylogarithmic in $n$. This extends a celebrated lower bound of Matou\v{s}ek for low distortion embeddings of random graphs into $\ell_q$ spaces.
\end{abstract}

\maketitle

\tableofcontents

\numberwithin{equation}{section}

\section{Introduction} \label{sec1}

\subsection{Expander graphs, and discrete Poincar\'{e} inequalities}

Let $d\geqslant 3$ be an integer, let $G=(V_G,E_G)$ be a $d\text{-regular}$ graph, and set $n:=|V_G|$. The adjacency matrix $A_G$ of\, $G$ is symmetric, its eigenvalues are real, and we write them as
\begin{equation} \label{eq-eigen1}
\lambda_n(G) \leqslant \cdots \leqslant \lambda_2(G) \leqslant \lambda_1(G)=d.
\end{equation}
We also set
\begin{equation} \label{eq-eigen2}
\lambda(G):=\max\big\{ |\lambda_2(G)|,\dots, |\lambda_n(G)|\big\}.
\end{equation}
The reciprocal of the spectral gap, and more precisely the quantity $\frac{d}{d-\lambda_2(G)}$, can also be defined as the smallest constant $\gamma\in (0,\infty]$ such that for any $f\colon V_G\to \mathbb{R}$,
\begin{equation} \label{poincare-real}
\frac{1}{|V_G|^2} \sum_{v,w\in V_G} \big(f(v)-f(w)\big)^2 \leqslant \frac{\gamma}{|E_G|}
\sum_{\{v,w\}\in E_G} \big(f(v)-f(w)\big)^2.
\end{equation}
It is natural to generalize this discrete Poincar\'{e} inequality for functions taking values in normed spaces or, more generally, in metric spaces. These generalizations have been studied extensively in metric geometry, nonlinear functional analysis and, more recently, in theoretical computer science; see, e.g., \cite{ANNRW18a,Es22,Gr03,Ma97,MN14,MN15,Pi10} and references therein.

In this paper we shall focus on extensions of \eqref{poincare-real} for functions taking values in normed spaces.

\begin{definition}[Nonlinear spectral gap]
Let $d\geqslant 3$ be an integer, and let $G=(V_G,E_G)$ be a $d\text{-regular}$ graph. Also let $X$ be a normed space, and let $p\geqslant 1$. By $\gamma(G,\|\cdot\|_X^p)$ we denote the smallest constant $\gamma\in (0,\infty]$ such that for any $f\colon V_G \to X$,
\begin{equation} \label{eq-poincare}
\frac{1}{|V_G|^2} \sum_{v,w\in V_G} \|f(v)-f(w)\|_X^p \leqslant \frac{\gamma}{|E_G|}
\sum_{\{v,w\}\in E_G} \|f(v)-f(w)\|_X^p.
\end{equation}
\end{definition}

Here, one is interested in obtaining estimates for $\gamma(G,\|\cdot\|_X^p)$ that depend on graph-parameters of $G$ and the geometry of $X$, but \emph{not} on the size of the vertex set of $G$. To this end, it is clear that one has to assume that the graph $G$ is an expander, that is, $G$ has a non-trivial spectral gap. It is also a standard observation that the space $X$ must not contain isomorphic copies of $\ell_{\infty}^k$'s of large dimension $k$. By a classical result of Maurey and Pisier \cite{MP76}, the latter property is quantified via the notion of cotype.

\begin{definition}[Cotype]
Let $q\geqslant 2$, and let\, $\mathrm{C}\geqslant 1$. We say that a normed space $X$ has \emph{cotype $q$ with constant $\mathrm{C}$} if for every choice $x_1,\dots,x_m$ of vectors in $X$,
\[ \Exp\,\Big\|\sum_{i=1}^m r_i x_i\Big\|_X^q \geqslant \frac{1}{\mathrm{C}^q}\sum_{i=1}^m \|x_i\|_X^q, \]
where $(r_i)_{i=1}^m$ are i.i.d. Rademacher random variables.
\end{definition}

The previous remarks can be summarized by saying that \eqref{eq-poincare} can only hold true for regular expander graphs and for normed spaces with non-trivial cotype. It is not known whether trivial cotype is the \emph{only} obstacle to the validity of \eqref{eq-poincare}.

\begin{problem}[Pisier/Mendel--Naor\footnote{Related problems have been asked by Ozawa \cite{Oz04} and Linial \cite[Paragraph 3.1]{Ma11}.}] \label{prob-pmn}
Is it true that every regular expander graph satisfies \eqref{eq-poincare} for functions taking values in a normed space $X$ with non-trivial cotype?
\end{problem}

Besides its intrinsic theoretical interest, a positive answer to Problem \ref{prob-pmn} will have profound consequences in analysis and geometry; we refer the reader to \cite[Section 1]{MN14} and \cite[Section~6]{Es22} for a detailed discussion of the scope of Problem \ref{prob-pmn}.

At present, Problem \ref{prob-pmn} remains open. In fact, even the existence of a single sequence $(G_n)$ of regular expander graphs with $|V_{G_n}|\to\infty$ that satisfy \eqref{eq-poincare} for all normed spaces with non-trivial cotype has not been established. The strongest result in this direction is due to Lafforgue \cite{La08} who constructed, with algebraic methods, a sequence $(G_n)$ of regular expander graphs with $|V_{G_n}|\to\infty$ and which satisfy \eqref{eq-poincare} for all normed spaces with non-trivial {\it type}; further examples have been obtained in \cite{MN14,dLdS23,Opp23}.

\subsubsection{Normed spaces with an unconditional basis}

Nonlinear spectral gaps are better understood (at least qualitatively) in the important setting of normed spaces with an unconditional basis, a standard form of symmetry that is frequently assumed in Banach space theory.

\begin{definition}[Unconditional bases]
Let $K\geqslant 1$. We say that a basis\footnote{Here, we implicitly assume that $X$ is finite-dimensional; the infinite-dimensional case is defined similarly.} $(x_j)_{j=1}^k$ of a normed space~$X$ is \emph{$K$-unconditional} if for every choice of scalars $a_1,\dots,a_k$,
\[ \max\bigg\{ \Big\|\sum_{j=1}^k \sigma_j a_j x_j \Big\|_X \colon \sigma_1,\dots,\sigma_k\in \{-1,1\}\bigg\}
\leqslant K \bigg\| \sum_{j=1}^k a_j x_j\bigg\|_X. \]
In particular, a basis $(x_j)_{j=1}^k$ is $1$-unconditional if for every choice of scalars $a_1,\dots,a_k$,
\[ \bigg\|\sum_{j=1}^k a_j x_j\bigg\|_X =\bigg\|\sum_{j=1}^k |a_j| x_j\bigg\|_X. \]
\end{definition}

The class of normed spaces with an unconditional basis and non-trivial cotype is quite extensive. It includes a variety of classical spaces, such as $\ell_q$ for all $q\geqslant 1$, and $L_q$ for all $q>1$, as well as more ``exotic" examples (e.g., \cite{FJ74}).

The Poincar\'{e} constant $\gamma(G,\|\cdot\|_X^p)$ for a normed space $X$ with an unconditional basis and non-trivial cotype can be estimated using a  transfer argument that originates in the work of Ozawa~\cite{Oz04}. More precisely, by combining a comparison result of nonlinear spectral gaps due to Naor \cite{Na14} with a nonlinear embedding due to Odell--Schlumprecht \cite{OS94}, one can show that if $G$ is a $d$-regular graph ($d\geqslant 3$), then for any normed space $X$ with a $1$-unconditional basis that has cotype $q\geqslant 2$
with constant $\mathrm{C}\geqslant 1$, we have\footnote{For details concerning our asymptotic notation we refer to Section \ref{sec2}.}
\begin{equation} \label{eq-os-intro-1}
\gamma(G,\|\cdot\|_X^2) \lesssim (q\mathrm{C})^{O(q)}\, \Big(\frac{d}{d-\lambda_2(G)}\Big)^8.
\end{equation}
Under the additional assumption that $X$ is $q$-concave with $M_{(q)}(X)=1$ (e.g., if $X$ is $\ell_q$), this estimate can be improved to
\begin{equation} \label{eq-os-intro-2}
\gamma(G,\|\cdot\|_X^2) \lesssim 2^{O(q)}\, \Big(\frac{d}{d-\lambda_2(G)}\Big)^8.
\end{equation}
We refer the reader to Appendix \ref{appendix} for information on the concepts not introduced at this point, as well as for an exposition of the proofs of \eqref{eq-os-intro-1} and \eqref{eq-os-intro-2}.

The exponential dependence on the cotype in \eqref{eq-os-intro-1} and \eqref{eq-os-intro-2} is an inherent characteristic of Ozawa's method. This has been analyzed in detail by Naor \cite{Na14} who also showed, by a completely different argument, an optimal estimate for the case of $\ell_q$ spaces:
\[ \gamma(G,\|\cdot\|_{\ell_q}^2) \lesssim q^2\, \Big(\frac{d}{d-\lambda_2(G)}\Big), \]
for any $d$-regular graph $G$ and any $q\geqslant 2$.

\subsection{Main results}

The goal of this paper is to obtain a nearly optimal improvement of \eqref{eq-os-intro-1} for random regular graphs.

To this end, the following combinatorial property of (deterministic) regular graphs is a principal conceptual contribution of this paper. (For any graph $G$, by $\dist_G(\cdot,\cdot)$ we denote the shortest-path distance on $G$; see Section \ref{sec2} for more details.)

\begin{definition}[Long-range expansion] \label{def-lre}
Let $n> d\geqslant 3$ be integers, and let $G$ be a $d$-regular graph on $[n]:=\{1,\dots,n\}$. Given parameters $\alpha\in(0,1]$, $\varepsilon\in(0,1]$ and $L\geqslant 1$, we say that $G$ satisfies the \emph{long-range expansion property}, denoted by $\mathrm{Expan}(\alpha,\varepsilon,L)$, if the following are satisfied.
\begin{enumerate}
\item[$(A)$] \label{Part-A} For every nonempty subset $S$ of $[n]$ and every positive integer $\ell$,
\[ \big|\{v\in [n]\colon \dist_G(v,S)\leqslant \ell\}\big|
\geqslant \min\Big\{\frac{3n}{4},\alpha (d-1)^\ell\,|S|\Big\}. \]
\item[$(B)$] \label{Part-B} For every nonempty subset $S$ of $[n]$ and every positive integer $\ell$ with
$\alpha (d-1)^{\ell-1}|S|\leqslant \frac{3n}{4}$, setting
\[ T:=\Big\{e\in E_G\colon \big|\{v\in S \colon \dist_G(v,e)\leqslant \ell-1\}\big|\geqslant L(d-1-\varepsilon)^\ell\Big\}, \]
there is at least one vertex $v\in S$ with
\[ \big|\{e\in T\colon \dist_G(v,e)\leqslant \ell-1\}\big|\leqslant L(d-1-\varepsilon)^\ell. \]
\end{enumerate}
\end{definition}

As we shall see in Proposition \ref{lemma:long-range-expansion} below, for any integer $d\geqslant 6$, a uniformly random $d$-regular graph satisfies property $\mathrm{Expan}(\alpha,\varepsilon,L)$---for appropriate choices of $\alpha,\varepsilon$ and $L$---with high probability. More importantly, the long-range expansion property has a significant effect on the behavior of the Poincar\'{e} constant. This is the content of the following theorem which is the main result of this paper.

\begin{theorem}[Nonlinear spectral gap via long-range expansion] \label{aekjfnakfn}
Let $q\geqslant 2$, let $\mathrm{C}\geqslant 1$, let $K\geqslant 1$, let $d\geqslant 3$ be an integer, let $\alpha,\varepsilon\in (0,1]$, let $L\geqslant 1$, and set
\begin{equation} \label{eq-Gamma}
\Gamma=\Gamma(q,\mathrm{C},K,d,\alpha,\varepsilon,L):=
2^{115}\, q^{10}\, \mathrm{C}^2\, K^3\, \Big( \frac{d^{25}\, L^8}{\alpha^{14}\, \varepsilon^{18}}\Big).
\end{equation}
Let $G$ be a $d$-regular graph that satisfies property $\mathrm{Expan}(\alpha,\varepsilon,L)$. Then, for any finite-dimensional normed space $X$ with a $K$-unconditional basis that has cotype $q$ with constant\, $\mathrm{C}$, we have
\begin{equation} \label{eq-new1}
\gamma(G,\|\cdot\|_X) \leqslant \Gamma;
\end{equation}
more generally, for any $p\geqslant 1$,
\begin{equation} \label{eq-newnew1}
\gamma(G,\|\cdot\|^p_X) \leqslant (p\Gamma)^p.
\end{equation}
\end{theorem}

\begin{remark} \label{rem-new1}
Fix $q\geqslant 2$ and an integer $d\geqslant 3$, and let $G$ be any $d$-regular graph on~$[n]$ with $\log n \gtrsim q$. By Matou\v{s}ek's version \cite{Ma97} of Bourgain's embedding theorem, there exists a map $f\colon [n]\to \ell_q^k$, with $k$ large enough, such that $\dist_G(v,w) \leqslant \|f(v)-f(w)\|_{\ell_q^k} \lesssim \frac{\log n}{q} \, \dist_G(v,w)$ for all $v,w\in [n]$, which is easily seen to imply that $\gamma(G,\|\cdot\|^p_{\ell^k_q}) \gtrsim_{d,p} q^p$ for any $p\geqslant 1$. We note that $\ell^k_q$ has a $1$-unconditional basis and cotype $q$ with constant $\mathrm{C}=O(1)$ (see, e.g., \cite{JL01}). Thus, our estimates in \eqref{eq-new1} and \eqref{eq-newnew1} establish that an optimal dependence of the Poincar\'e constant on $q$ is {\it polynomial}.
\end{remark}

\begin{remark} \label{rem-new2}
The bulk of the proof of Theorem \ref{aekjfnakfn} consists of showing the estimate \eqref{eq-new1}. The case $p>1$ follows easily from the case $p=1$ and a (standard) variant of Matou\v{s}ek's extrapolation argument \cite{Ma97}.
\end{remark}

\begin{remark} \label{rem-lattices}
Theorem \ref{aekjfnakfn} can be used to estimate the Poincar\'e constant $\gamma(G,\|\cdot\|_X)$ for any $d\text{-regular}$ graph $G$ which satisfies property $\mathrm{Expan}(\alpha,\varepsilon,L)$ and any (possibly infinite-dimensional) Banach lattice $X$ which has cotype $q$ with constant\, $\mathrm{C}$. Indeed, by a classical result of Maurey \cite[Proposition 1]{Mau74}, any finite-dimensional subspace $F$ of a Banach lattice $X$ is contained in a finite-dimensional subspace $Y$ of $X$ with a $2$-unconditional basis; consequently, by Theorem \ref{aekjfnakfn}, we have $\gamma(G,\|\cdot\|_X) \leqslant \Gamma(q,\mathrm{C},2,d,\alpha,\varepsilon,L)$.
\end{remark}

The following proposition complements Theorem~\ref{aekjfnakfn} and shows that regular graphs with the long-range expansion property exist in abundance.

\begin{proposition}[Long-range expansion is typical] \label{lemma:long-range-expansion}
Let $d \geqslant 3$ be an integer, and set
\begin{align}
\alpha=\alpha(d) & := d^{-10^{11} \ln d},\label{era-1} \\
\varepsilon & := 0.2, \label{era-2} \\
L= L(d) & := 24\,\alpha^{-1}. \label{era-3}
\end{align}
Let $n>d$ such that $dn$ is an even integer, and let\, $\mathbb{P}$ be the uniform probability measure~on~$G(n,d)$,
the set of $d$-regular graphs on $n$ vertices.
\begin{enumerate}
\item[(a)] There exists a constant $c_1=c_1(d)>0$ such that
\[ \mathbb{P}\big[G \text{ satisfies part \emph{(\hyperref[Part-A]{$A$})} of property\, $\mathrm{Expan}(\alpha,\varepsilon,L)$}\big]
\geqslant 1-O_{d}\Big(\frac{1}{n^{c_1}}\Big);  \]
namely, for any integer $d\geqslant 3$, a uniformly random $d$-regular graph on $n$ vertices satisfies
part \emph{(\hyperref[Part-A]{$A$})} of the long-range expansion property\, $\mathrm{Expan}(\alpha,\varepsilon,L)$
with high probability.
\item[(b)] If\, $d\geqslant 6$ and $G$ is any (deterministic) $d$-regular graph on $[n]$ with
$\lambda(G) \leqslant 2.1\sqrt{d-1}$, then $G$ satisfies part \emph{(\hyperref[Part-B]{$B$})}
of property\, $\mathrm{Expan}(\alpha,\varepsilon,L)$. Thus, by Friedman's
second eigenvalue theorem \emph{\cite{Fr08}}, for any integer $d\geqslant 6$, there exists a constant $c_2=c_2(d)>0$ such that
\[ \mathbb{P}\big[G \text{ satisfies part \emph{(\hyperref[Part-B]{$B$})} of property\, $\mathrm{Expan}(\alpha,\varepsilon,L)$}\big]
\geqslant 1-O_{d}\Big(\frac{1}{n^{c_2}}\Big).  \]
\end{enumerate}
In particular, for any integer $d\geqslant 6$, a uniformly random $d$-regular graph on $n$ vertices
satisfies the long-range expansion property  $\mathrm{Expan}(\alpha,\varepsilon,L)$
with probability at least $1-O_d\left(n^{-c_3}\right)$
for some constant $c_3=c_3(d)>0$.
\end{proposition}

Evidently part (\hyperref[Part-B]{$B$}) of the long-range expansion property is a spectral property related to edge expansion. On the other hand, part (\hyperref[Part-A]{$A$}) quantifies vertex expansion. A closely related property called \emph{lossless expansion}\footnote{A graph $G=(V_G,E_G)$ is said to be an \emph{$(\eps,\delta)$-lossless expander} if, for every nonempty $S \subseteq V_G$ with $|S| \leqslant \eps |V_G|$, we have that $\big| \{v\in V_G \colon \dist_G(v,S)= 1\}\big| \geqslant (d-1-\delta)|S|$.} has been extensively studied
in the computer science literature (see, e.g., \cite{Ka95,CRVW02}), and is known to hold for random regular graphs of constant degree~\cite{Vad12}.

\begin{remark}[Explicit constructions] \label{rem-explicit-constructions}
Part (\hyperref[Part-B]{$B$}) of the long-range expansion property is satisfied by all explicit constructions of Ramanujan graphs \cite{LPS88}. On the other hand, it seems quite challenging to give explicit constructions of regular graphs satisfying part (\hyperref[Part-A]{$A$}). Indeed, the closely related problem of finding explicit constructions of lossless expanders remains open and it has profound applications in computer science, especially to coding theory. The explicit construction of either of these forms of vertex expansion could plausibly lead to the other. As a natural candidate, it is a folklore conjecture that the Ramanujan graphs constructed
in \cite{LPS88} should be lossless expanders. However, the matter is delicate: there are high-girth Ramanujan graphs which are known to be far from having lossless expansion \cite{MS21}.
\end{remark}

\begin{remark}[Update, Oct. 2026]
Explicit sequences of graphs satisfying long-range expansion have been constructed, and in fact any sequence of Ramanujan graphs has the long-range expansion property. Additionally, both parts of Proposition~\ref{lemma:long-range-expansion} were extended to all $d \geqslant 3$. These developments are in a work of the first named author, roughly two years after the initial appearance of this article. Proposition~\ref{lemma:long-range-expansion} has also been independently extended to $d\geqslant 3$ around the same time by A. R. Polavarapu.
\end{remark}

\subsection{Application to bi-Lipschitz embeddings}

Given a finite metric space $\mathcal{M}=(M,d)$ and a normed space $(X,\|\cdot\|_X)$, the \emph{$($bi-Lipschitz$)$ distortion} of $\mathcal{M}$ into $X$, denoted by $c_X(\mathcal{M})$, is defined to be the smallest positive constant $D$ for which there exists a map $f\colon M\to X$ such that $d(i,j)\leqslant \|f(i)-f(j)\|_X\leqslant D\, d(i,j)$ for all $i,j\in M$. When $X=\ell_p$ with $p\in [1,\infty]$, then for simplicity we set $c_p(\mathcal{M})= c_{\ell_p}(\mathcal{M})$. It is a classical fact that $c_{\infty}(\mathcal{M})=1$, so we may define $p(\mathcal{M})$ to be the infimum over all $p\in [2,\infty)$ such that $c_p(\mathcal{M})<20$. (As it is explained in \cite{Na14}, the choice of the number $20$ is rather arbitrary, and it is fixed for convenience; one could equally define the parameter $p(\mathcal{M})$ using any number bigger than $1$ instead of $20$.) In particular, for any connected\footnote{If $G$ is not connected, then, by convention, we set $p(G)=0$.} $d$-regular graph $G$ on $[n]$, $p(G)$ denotes this parameter for the metric space $([n],\mathrm{dist}_G)$.

The study of the statistical properties of $p(G)$, when $G$ is sampled uniformly from $G(n,d)$, can be traced to the work of Matou\v{s}ek \cite{Ma97} that implies that, in the regime $d=O(1)$, we have
\begin{equation} \label{eq-dist-e1}
p(n,d):= \underset{G\sim G(n,d)}{\mathbb{E}}\, p(G) \gtrsim \log_d n.
\end{equation}
More refined estimates for $p(n,d)$ have been obtained by Naor \cite[Proposition 1.9]{Na14}.

One can consider a relaxation of the parameter $p(G)$ by replacing $\ell_p$ spaces with normed spaces with an unconditional basis and cotype. More precisely, for any connected $d$-regular graph $G$~on~$[n]$, let $\mathrm{uc}(G)$ denote the infimum over all $q\in [2,\infty)$ such that $c_X(G)<20$ for some finite-dimensional normed space $X$ that has a $20$-unconditional basis and cotype $q$ with constant~$20$. (Again, $20$ is rather arbitrary, and it is fixed for convenience.) Notice that $\mathrm{uc}(G)\leqslant p(G)$.

As a consequence of the main result of our paper, we have the following analogue of \eqref{eq-dist-e1}.

\begin{corollary}
For any integer $d\geqslant 6$,
\[ \underset{G\sim G(n,d)}{\mathbb{E}}\, \mathrm{uc}(G) \gtrsim_d (\log_d n)^{\frac{1}{10}}; \]
in other words, the quantity $\mathbb{E}\, \mathrm{uc}(G)$ grows polylogarithmically in the vertex set size.
\end{corollary}

\begin{proof}
Fix a $d$-regular graph $G$ on $[n]$ that satisfies property $\mathrm{Expan}(\alpha,\varepsilon,L)$ for some $\alpha,\varepsilon,L>0$. Set $q:=2\,\mathrm{uc}(G)$ and notice that, by definition of $q$, there exists a finite-dimensional normed space $X$ that has a $20\text{-unconditional}$ basis and cotype $q$ with constant $20$, and a map $f\colon [n]\to X$ such that $\dist_G(v,w)\leqslant \|f(v)-f(w)\|_X\leqslant 20\, \dist_G(v,w)$ for all $v,w\in [n]$. Since $G$ is $d$-regular, we have
\[ \log_d n \lesssim_d \frac{1}{n^2} \sum_{v,w\in [n]} \dist_G(v,w) \leqslant \frac{1}{n^2} \sum_{v,w\in [n]} \|f(v)-f(w)\|_X. \]
On the other hand, by Theorem \ref{aekjfnakfn}, we have
\begin{align*}
\frac{1}{n^2} \sum_{v,w\in [n]} \|f(v)-f(w)\|_X & \leqslant
20\,\gamma(G,\|\cdot\|_X) \, \frac{1}{|E_G|}\, \sum_{\{v,w\}\in E_G} \dist_G(v,w) \\
& \lesssim_{d,\alpha,\varepsilon,L} q^{10}
\lesssim_{d,\alpha,\varepsilon,L} \mathrm{uc}(G)^{10},
\end{align*}
that implies that $\mathrm{uc}(G)\gtrsim_{d,\alpha,\varepsilon,L} (\log_d n)^{\frac{1}{10}}$. By Proposition \ref{lemma:long-range-expansion}, the result follows.
\end{proof}

\subsection{Outline of the proof of Theorem \ref{aekjfnakfn}}

The first step is to reduce the validity of the Poincar\'{e} inequality  \eqref{eq-poincare} to the estimate
\begin{equation} \label{expo-new}
\sum_{v\in V_G} \|f(v)\|_X \lesssim \sum_{\{w,w'\}\in E_G} \|f(w)-f(w')\|_X,
\end{equation}
for functions $f\colon V_G\to \{-1,0,1\}^k$ which map to ``binary'' vectors and have zero as their empirical median. The advantage of this reduction is that, by the assumption that $X$ has an unconditional basis, the norm of a binary vector is controlled by the location of its support; in particular, for an edge $e=\{w,w'\}\in E_G$ and a vertex $v\in V_G$, the vector $f(w)-f(w')$ has bigger norm than the vector $f(v)$ if the support of $f(w)-f(w')$ contains the support of~$f(v)$. This reduction is established in Corollary \ref{cor-strong-poincare} by noting that, if a \textit{dimension-free} non-linear Poincar\'e inequality for functions taking binary values holds true, then the general case also follows by encoding arbitrary vectors using binary vectors, possibly in some higher-dimensional space.

With this reduction at hand we proceed as follows. For every vertex $v\in V_G$ and every coordinate~$j$ in the support of $f(v)$, it is shown in Lemma \ref{afhbaofjhbojhb} that there is some distance $\ell_{v,j}$ such that a significant portion of edges $e=\{w,w'\}\in B(v,\ell_{v,j}-1)$ (that is, edges at distance at most $\ell_{v,j}-1$ from $v$) satisfy $f(w)_j-f(w')_j \neq 0$. This is a consequence of
part (\hyperref[Part-A]{$A$}) of the long-range expansion property and the assumption that $f$ has zero as an empirical median. Next, we extract a large subset of these edges, that we denote by $\tilde E(v,j)$, so that $\tilde E(v,j)$ is still a significant portion of $B(v,\ell_{v,j}-1)$ and, moreover, the collection $(\tilde E(v,j))_v$ is nearly disjoint for every fixed $j$. This is achieved in Lemma~\ref{afjhsbfiasdjfhbdsifjhb} using part (\hyperref[Part-B]{$B$}) of the long-range expansion property.

For each distance $\ell$ and each vertex $v\in V_G$, let $J(v,\ell):= \{j\colon \ell_{v,j} = \ell\}$ denote the ``scale $\ell$ support'' of $f(v)$. Moreover, for an edge $e=\{w,w'\}\in E_G$, let $J(v,\ell,e):= \{j \in J(v,\ell)\colon e \in \tilde E(v,j)\}$ denote the ``scale $\ell$ mutual support" of $f(w)-f(w')$ and $f(v)$. In Lemma~\ref{aifgvsifghvfihfgiw} we show that for each scale $\ell$ and each vertex $v\in V_G$, we can find a set of edges whose size is of order $(d-1)^\ell$ such that for each edge in this set, $\|\chi_{J(v,\ell,e)}\|_X$ is comparable to $\|\chi_{J(v,\ell)}\|_X$; we informally say that such an edge \emph{compensates} $v$ at scale~$\ell$.

In order to finish the proof, we combine our analysis along all scales $\ell$. To this end, we write
\begin{equation} \label{eq-recombination}
\sum_{v \in V_G} \|f(v)\|_X =
\sum_{v \in V_G} \Big\|\sum_{\ell\geqslant 1} P_{J(v,\ell)}\big(f(v)\big)\Big\|_X
\leqslant \sum_{\ell\geqslant 1} \sum_{v \in V_G}  \big\| P_{J(v,\ell)}\big(f(v)\big)\big\|_X\,.
\end{equation}
Thus, it is enough to bound the right-hand side of \eqref{eq-recombination} by the contribution of edges in the right-hand side of \eqref{expo-new}. Our previous estimates only allow us to achieve this for a single scale $\ell$. We claim, however, that these estimates can be strengthened to be summable over~$\ell$. The mechanism is the following dichotomy, made precise in Lemma \ref{afokjnlfkjnsadfkjhsfb}. As previously noted, each vertex has order $(d-1)^\ell$ compensating edges at scale $\ell$.
The key issue is how often these edges are reused between the vertices. If the total number of edges that are used to compensate vertices at a given scale is sufficiently large, then the estimates coming from Lemma~\ref{aifgvsifghvfihfgiw} together with a double-counting argument can be used to establish the summability over~$\ell$. Otherwise, for the remaining scales, we can extract a smaller, but still sizable, set of edges each of which must be used to compensate many vertices. Since the sets $\tilde E(v,j)$ are nearly disjoint (over~$v$, for fixed $j$), the only way for an edge $e=\{w,w'\}\in E_G$ to be used to compensate many vertices at a given scale is for the vector $f(w)-f(w')$ to have very large support. And in this case, the bounds obtained by Lemma~\ref{aifgvsifghvfihfgiw} can be significantly strengthened. This is achieved using Proposition~\ref{aljfhbsjfhbskajhsb} which relies on the fact that the space $X$ has an unconditional
basis and non-trivial cotype. In either case, this dichotomy yields ~\eqref{expo-new}, completing the proof of Theorem \ref{aekjfnakfn}.

\section{Background material} \label{sec2}

\subsection{Graph theoretic preliminaries}

All graphs in this paper are simple; we allow multigraphs to have self-loops. If $G$ is any graph, then by $V_G$ we denote its vertex set, and by $E_G$ we denote the set of its edges. As we have already mentioned, for any pair $n\geqslant d\geqslant 3$ of positive integers, by $G(n,d)$ we denote the set of all $d$-regular graphs on $[n]:=\{1,\dots,n\}$.

\subsubsection{\!}

For any graph/multigraph $G$, by $\dist_G(\cdot,\cdot)$ we denote the shortest-path distance on~$G$ between vertices; by convention we set $\dist_G(v,w):=|V_G|$ if $v,w\in V_G$ are contained in different connected components of $G$. If $v$ is a vertex of $G$ and $S$ is a nonempty subset of $V_G$, then define the shorthand $\dist_G(v,S):= \min\big\{ \dist_G(v,w)\colon w\in S\big\}$; in particular, if $e = \{w,w'\}$ is an edge of $G$, then $\dist_G(v,e):=\min\big\{ \dist_G(v,w),\dist_G(v,w')\big\}$.

\subsubsection{\!}

If $X$ is a finite set, then a \emph{(partial) matching} of $X$ is a collection of pairwise (vertex-)disjoint edges of the complete graph on $X$; notice, in particular, that the empty set is always a (partial) matching. A \emph{perfect matching} of $X$ is a matching such that every element of $X$ belongs to some edge of the matching; note that there exist perfect matchings of $X$ if and only if $|X|$ is even.

\subsection{Banach space theoretic preliminaries}

The standard basis vectors in $\R^k$ will be denoted by $e_1,\dots,e_k$. Given a vector $y\in\R^k$, by $\supp(y):=\{j\in [k]\colon y_j\neq 0\}$ we denote its support, and by $|y|\in \mathbb{R}^k$ we denote the vector of absolute values of components of $y$; moreover, if $J\subseteq [k]$, then we set $P_J(y):=\sum_{s\in J}\langle y,e_s\rangle e_s$, so that $P_J(y)$ is obtained from $y$ by zeroing out all components indexed by $J^\complement:= [k]\setminus J$. Finally, for any $J \subseteq [k]$, we define the vector $\chi_J \in \R^k$ by
$\chi_J:= (\mathbbm{1}_{[j \in J]})_{j=1}^k$.

\begin{convention}
With a slight abuse of terminology, we will say that a norm $\|\cdot\|_X$ in $\R^k$ is $1\text{-unconditional}$ if the standard vector basis in $\R^k$ is $1$-unconditional, that is, if\, $\|y\|_X =\big\||y|\big\|_X$ for all $y\in\R^k$.
\end{convention}

\subsection{Asymptotic notation}

If $a_1,\dots,a_k$ are parameters, then we write $O_{a_1,\dots,a_k}(X)$ to denote a quantity that is bounded in magnitude by $X C_{a_1,\dots,a_k}$, where $C_{a_1,\dots,a_k}$ is a positive constant that depends on the parameters $a_1,\dots,a_k$. We also write $Y\lesssim_{a_1,\dots,a_k}\!X$ or $X\gtrsim_{a_1,\dots,a_k}\!Y$ for the estimate $|Y|=O_{a_1,\dots,a_k}(X)$. Finally, we write $o_{a_1,\dots,a_k}(1)$ to denote a sequence that goes to $0$ with a rate of convergence depending on $a_1,\dots,a_k$.

\section{Restricted cotype, and reductions} \label{sec3}

\begin{definition}[Restricted cotype]
Let $q \geqslant 2$, and let $\mathrm{C}\geqslant 1$. We say that a finite collection $\{y_\alpha\}_{\alpha\in A}$ of vectors in a Banach space $X$ has \emph{restricted cotype $q$ with constant $\mathrm{C}$} if for every nonempty subset $A'$ of $A$,
\[ \Exp\,\Big\|\sum_{\alpha\in A'} r_\alpha y_\alpha\Big\|_X^q
\geqslant \frac{1}{\mathrm{C}^q} \sum_{\alpha\in A'} \|y_\alpha\|_X^q, \]
where $(r_\alpha)_{\alpha\in A}$ are i.i.d. Rademacher random variables.
\end{definition}

We will apply the definition of restricted cotype to reduce the proof of Theorem~\ref{aekjfnakfn} to the setting where the components of the vectors $f(v)$, $v\in[n]$, take values in $\{-1,0,1\}$.

\begin{theorem}[Poincar\'{e} inequality for vectors in $\{-1,0,1\}^k$]\label{afjnobgiugbo}
Let $q\geqslant 2$, let $\mathrm{C}\geqslant 1$, let $d\geqslant 3$ be an integer, let $\alpha,\varepsilon\in (0,1]$, let $L\geqslant 1$, and set
\begin{equation} \label{eq-new2}
\Pi=\Pi(q,\mathrm{C},d,\alpha,\varepsilon,L) := 2^{113}\, q^{10}\, \mathrm{C}^2\,
\Big( \frac{d^{24}\, L^8}{\alpha^{14}\,\varepsilon^{18}}\Big).
\end{equation}
Let $G$ be a $d$-regular graph on $[n]$ that satisfies property $\mathrm{Expan}(\alpha,\varepsilon,L)$, let $k\geqslant 1$, and let $\|\cdot\|_X$ be an $1\text{-unconditional}$ norm in $\R^k$. If $f\colon [n]\to \{-1,0,1\}^k$ is any non-identically zero map satisfying
\begin{equation} \label{eq-new-empirical}
\min\Big\{ \big|\{v\in [n]\colon f(v)_j\geqslant 0\}\big|,
\big|\{v\in [n]\colon f(v)_j\leqslant 0\}\big| \Big\}\geqslant \frac{n}{2} \ \ \ \text{ for every } j\in [k],
\end{equation}
and such that the collection of vectors $\big\{P_J\big(|f(v)|\big)\colon v\in [n] \text{ and }J\subseteq [k]\big\}$ has restricted cotype $q$ with constant $\mathrm{C}$, then
\[ \sum_{v\in[n]}\|f(v)\|_X \leqslant \Pi\,\sum_{\{w,w'\}\in E_G}\|f(w)-f(w')\|_X. \]
\end{theorem}

Theorem~\ref{afjnobgiugbo} is proved in the next section. Here, we prove some of its consequences. We start with the following corollary.

\begin{corollary} \label{cor-strong-poincare}
Let the parameters $q,\mathrm{C},d,\alpha,\varepsilon,L$ be as in Theorem~\ref{afjnobgiugbo}, and let\, $\Pi$ be as in \eqref{eq-new2}. Let $G$ be a $d$-regular graph on $[n]$ that satisfies property $\mathrm{Expan}(\alpha,\varepsilon,L)$, let $k\geqslant 1$, and let $\|\cdot\|_X$ be a $1$-unconditional norm in $\R^k$ that has cotype $q$ with constant\, $\mathrm{C}$. If $f\colon [n]\to \R^k$ is any non-constant map such that $0$ is an empirical median of $f$ in the sense of \eqref{eq-new-empirical}, then
\begin{equation} \label{ajhgfvaeiyfvewifywv}
\sum_{v\in[n]}\|f(v)\|_X \leqslant \frac{3\Pi}{2} \sum_{\{w,w'\}\in E_G}\|f(w)-f(w')\|_X;
\end{equation}
moreover, for any $p>1$,
\begin{equation} \label{eq-strong-poincare-p}
\sum_{v\in[n]}\|f(v)\|^p_X \leqslant \Big(\frac{3 \Pi}{2}\Big)^p\, d^{p-1} \big(\max\{2,p\}\big)^p
\sum_{\{w,w'\}\in E_G}\|f(w)-f(w')\|^p_X.
\end{equation}
\end{corollary}

\begin{proof}
We first give the proof of \eqref{ajhgfvaeiyfvewifywv}. We fix the map $f$, and we define
\[ \delta:=\frac{1}{4}\min\big\{|f(v)_j-f(w)_j|\colon v,w\in[n] \text { and } j\in [k] \text{ with } f(v)_j\neq f(w)_j\big\}, \]
and
\[ m:=\Big\lceil\frac{1}{\delta} \max\big\{|f(v)_j|\colon v\in[n] \text{ and } j\in [k]\big\} \Big\rceil. \]
Further, define a norm $\|\cdot\|_{X(\ell_1^m)}$ in $\R^{km}=\prod_{j=1}^k \R^m$ by
\[ \|(y_1,\dots,y_k)\|_{X(\ell_1^m)}:= \big\|(\|y_1\|_{\ell_1^m},\dots,\|y_k\|_{\ell_1^m})\big\|_X
\ \ \ \text{for every} \ (y_1,\dots,y_k)\in \prod_{j=1}^k \R^m. \]
For every $v\in[n]$, define a vector $\tilde f(v)=\big(\tilde f_1(v),\dots,\tilde f_k(v)\big)$ in $\prod_{j=1}^k \R^m$ as follows.
\begin{itemize}
    \item If $f(v)_j=0$, then we set $\tilde f_j(v)=0$.
    \item If $f(v)_j>0$, then let $\tilde f_j(v)$ be the $\{0,1\}$--vector in $\R^m$ whose first $\big\lceil\frac{f(v)_j}{\delta}\big\rceil$ components are equal to $1$ and the remaining components are equal to zero.
    \item If $f(v)_j<0$, then, as above, let $\tilde f_j(v)$ be the $\{0,-1\}$--vector in $\R^m$ whose first
    $\big\lceil\frac{-f(v)_j}{\delta}\big\rceil$ components are equal to $-1$ and the remaining components are equal to zero.
\end{itemize}
Note that, with the above definition, whenever $f(v)_j\neq 0$, we have
\[ \|\tilde f_j(v)\|_{\ell_1^m}\geqslant \frac{1}{\delta}|f(v)_j|, \]
implying that
\[ \delta \sum_{v\in[n]}\|\tilde f(v)\|_{X(\ell_1^m)} \geqslant
\delta \sum_{v\in[n]}\Big\| \Big(\frac{1}{\delta}|f(v)_j|\Big)_{j\leqslant k}\Big\|_X = \sum_{v\in[n]}\|f(v)\|_X. \]
Further, whenever $f(w)_j\neq f(w')_j$, we have
$$
\frac{1}{\delta}|f(w)_j-f(w')_j|
\geqslant
\|\tilde f_j(w)-\tilde f_j(w')\|_{\ell_1^m}-2
\geqslant
\|\tilde f_j(w)-\tilde f_j(w')\|_{\ell_1^m}
-\frac{1}{2}\cdot \frac{1}{\delta}|f(w)_j-f(w')_j|,
$$
implying that
$$ \delta\sum_{\{w,w'\}\in E_G}\|\tilde f(w)-\tilde f(w')\|_{X(\ell_1^m)}
\leqslant \frac{3}{2}\sum_{\{w,w'\}\in E_G}\|f(w)-f(w')\|_X. $$
Thus, in order to prove \eqref{ajhgfvaeiyfvewifywv},
it is enough to verify that
\begin{equation}\label{eq005}
\sum_{v\in[n]}\|\tilde f(v)\|_{X(\ell_1^m)} \leqslant \Pi
\sum_{\{w,w'\}\in E_G}\|\tilde f(w)-\tilde f(w')\|_{X(\ell_1^m)}.
\end{equation}
Notice that, by \eqref{eq-new-empirical} and the definition of $\tilde f$,
$$ \min\Big\{\big|\{v\in [n]\colon \tilde f(v)_j\geqslant 0\}\big|,
\big|\{v\in [n]\colon \tilde f(v)_j\leqslant 0\}\big|\Big\}\geqslant
\frac{n}{2} \ \ \ \text{ for every } j\in [km]. $$
Hence, by Theorem~\ref{afjnobgiugbo}, in order to prove \eqref{eq005}, it suffices to show that
the collection of vectors $\big\{P_J\big(|\tilde f(v)|\big)\colon v\in[n] \text{ and } J\subseteq [km]\big\}$
has restricted cotype $q$ with constant $\mathrm{C}$ with respect to the norm~$\|\cdot\|_{X(\ell_1^m)}$.
To that end, take any collection of vectors $y^{(t)}=(y_1^{(t)},\dots,y_k^{(t)})\in\prod_{j=1}^k \R^{m}$,
$1\leqslant t\leqslant T$, with the property that for every $t\in [T]$ and every $j\in [k]$,
all components of $y_j^{(t)}\in\R^m$ are nonnegative (in particular, any coordinate projection of
$|\tilde f(v)|$, $v\in[n]$, satisfies this property), so that
$$ \sum_{s=1}^m\langle y_j^{(t)},e_s\rangle =\|y_j^{(t)}\|_{\ell_1^m}. $$
Then, for every $j\in [k]$, we have
$$ \Big\|\sum_{t=1}^T r_t\, y_j^{(t)}\Big\|_{\ell_1^m}
\geqslant \Big| \sum_{t=1}^T r_t \sum_{s=1}^m\langle y_j^{(t)},e_s\rangle \Big|
=\Big|\sum_{t=1}^T r_t\, \|y_j^{(t)}\|_{\ell_1^m} \Big|. $$
Thus, everywhere on the probability space,
$$ \Big\|\sum_{t=1}^T r_t\, y^{(t)} \Big\|_{X(\ell_1^m)}
=\Big\|\Big(\Big\|\sum_{t=1}^T r_t\, y_j^{(t)}\Big\|_{\ell_1^m}\Big)_{j\leqslant k}
\Big\|_{X} \geqslant \Big\|\Big(\sum_{t=1}^T r_t\, \|y_j^{(t)}\|_{\ell_1^m}\Big)_{j\leqslant k} \Big\|_{X}. $$
Since $X$ has cotype $q$ with constant $\mathrm{C}$, we conclude that
$$ \Exp\,\Big\|\sum_{t=1}^T r_t\, y^{(t)} \Big\|_{X(\ell_1^m)}^q \!
\geqslant \Exp\,\Big\|\Big(\sum_{t=1}^T r_t\, \|y_j^{(t)}\|_{\ell_1^m}\Big)_{j\leqslant k} \Big\|_{X}^q
\geqslant \frac{1}{\mathrm{C}^q} \sum_{t=1}^T \big\|\big( \|y_j^{(t)}\|_{\ell_1^m}\big)_{j\leqslant k}\big\|_X^q
=\frac{1}{\mathrm{C}^q} \sum_{t=1}^T \big\|y^{(t)}\big\|_{X(\ell_1^m)}^q. $$
This shows, in particular, that the family of coordinate projections of the vectors $|\tilde f(v)|$, $v\in[n]$,
has restricted cotype $q$ with constant $\mathrm{C}$ with respect to the norm $\|\cdot\|_{X(\ell_1^m)}$
and, consequently, the proof of \eqref{ajhgfvaeiyfvewifywv} is completed.

As we have already remarked, the estimate \eqref{eq-strong-poincare-p} follows from \eqref{ajhgfvaeiyfvewifywv}
and a variant of Matou\v{s}ek's extrapolation argument \cite{Ma97} adapted to the present setting;
as a consequence, our bound is slightly better than what is known in the general case (see, e.g.,
\cite[Lemma 4.4]{NS11} or \cite[Section~3]{Es22}). We proceed to the details.
Again, we fix the map $f$, and we define $g\colon [n]\to X$ by $g(v):= \|f(v)\|_X^{p-1}f(v)$ for every
$v\in[n]$. Clearly, $g$ is non-constant and satisfies \eqref{eq-new-empirical} and so, by~\eqref{ajhgfvaeiyfvewifywv},
\begin{equation} \label{eqkt031}
\sum_{v\in[n]}\|f(v)\|_X ^p =\sum_{v\in[n]}\|g(v)\|_X \leqslant \frac{3\Pi}{2} \sum_{\{v,w\}\in E_G}\|g(v)-g(w)\|_X.
\end{equation}
Fix a linear order $\prec$ on $[n]$ that satisfies $\|f(v)\|_X \leqslant \|f(w)\|_X$ if $v\prec w$. Then,
\begin{align} \label{eqkt033}
& \sum_{\{v,w\}\in E_G}  \|g(v)-g(w)\|_X =
\sum_{\substack{v,w\in [n]\\ \{v,w\}\in E_G, v\prec w}}
\big\| \|f(w)\|_X^{p-1}f(w)-\|f(v)\|_X^{p-1}f(v)\big\|_X \\
& \!\!\!\! \leqslant \sum_{\substack{v,w\in [n]\\ \{v,w\}\in E_G, v\prec w}} \!\!\!
\|f(v)\|_X^{p-1} \cdot \|f(w)-f(v)\|_X + \!\!\!
\sum_{\substack{v,w\in[n]\\ \{v,w\}\in E_G, v\prec w}}
\!\!\!\|f(w)\|_X  \cdot \big|\|f(w)\|_X^{p-1} - \|f(v)\|_X^{p-1} \big|. \nonumber
\end{align}
By H\"{o}lder's inequality, the first term of the right-hand side of \eqref{eqkt033} is bounded by
\begin{equation} \label{eqkt034}
\sum_{\substack{v,w\in [n]\\ \{v,w\}\in E_G, v\prec w}}
\|f(v)\|_X^{p-1} \cdot \|f(w)-f(v)\|_X \leqslant
\Big( d \sum_{v \in [n]} \|f(v)\|_X^p \Big)^{1-\frac{1}{p}}
\Big( \sum_{\{v,w\}\in E_G} \|f(v)-f(w)\|_X^p \Big)^{\frac{1}{p}}.
\end{equation}
Similarly, the second term of the right-hand side of \eqref{eqkt033} is estimated by
\begin{align}
& \!\!\!\sum_{\substack{v,w\in [n]\\ \{v,w\}\in E_G, v\prec w}}
\|f(w)\|_X \cdot \big|\|f(w)\|_X^{p-1} - \|f(v)\|_X^{p-1} \big| \label{eqkt035} \\
& \ \ \ \ \ = \sum_{\substack{v,w\in[n]\\ \{v,w\}\in E_G, v\prec w}}
\|f(w)\|_X ^{p-1} \Big( \|f(w)\|_X^{2-p} \cdot \big|  \|f(w)\|_X^{p-1} - \|f(v)\|_X^{p-1} \big| \Big) \nonumber \\
& \ \ \ \ \ \leqslant \Big( d \sum_{v \in [n]} \|f(v)\|_X^p \Big)^{1-\frac{1}{p}}
\Bigg( \sum_{\substack{v,w\in[n]\\ \{v,w\}\in E_G, v\prec w}}\!\!\!\!
\Bigg| \|f(w)\|_X - \|f(v)\|_X \Big(\frac{\|f(v)\|_X}{\|f(w)\|_X}\Big)^{p-2} \Bigg|^p \Bigg)^{\frac{1}{p}}. \nonumber
\end{align}
Using the elementary inequality
\[ \Big|a-b\Big(\frac{b}{a}\Big)^{p-2}\Big| \leqslant \max\{1,p-1\}|a-b| \ \ \ \text{ for every } a\geqslant b\geqslant 0, \]
by \eqref{eqkt035}, we obtain that
\begin{align}
\sum_{\substack{v,w\in [n]\\ \{v,w\}\in E_G, v\prec w}} &
\|f(w)\|_X \cdot \big|\|f(w)\|_X^{p-1} - \|f(v)\|_X^{p-1} \big| \label{eqkt036} \\
& \leqslant \max\{1,p-1\} \Big( d \sum_{v \in [n]} \|f(v)\|_X^p \Big)^{1-\frac{1}{p}}
\Big( \sum_{\{v,w\}\in E_G} \|f(v) - f(w)\|_X^p \Big)^{\frac{1}{p}}. \nonumber
\end{align}
Combining \eqref{eqkt031}--\eqref{eqkt034} and \eqref{eqkt036}, we conclude that
\[ \sum_{v\in[n]}\|f(v)\|_X ^p \leqslant \frac{3\Pi}{2}\, d^{\frac{p-1}{p}} \big(\max\{1,p-1\}+1\big)\,
\Big( \sum_{v \in [n]} \|f(v)\|_X^p \Big)^{1-\frac{1}{p}}
\Big( \sum_{\{v,w\}\in E_G} \|f(v) - f(w)\|_X^p \Big)^{\frac{1}{p}}, \]
which is clearly equivalent to \eqref{eq-strong-poincare-p}.
\end{proof}
We are now ready to give the proof of Theorem~\ref{aekjfnakfn}.
\begin{proof}[Proof of Theorem~\ref{aekjfnakfn}]
We start with a standard renorming. Specifically, fix a $K$-unconditional basis $(x_j)_{j=1}^k$ of $X$
and define an equivalent norm $\seminorm{\cdot}$ on $X$ by
\[ \seminorm{\sum_{j=1}^k a_j x_j}=
\max\bigg\{ \Big\|\sum_{j=1}^k \sigma_j a_j x_j\Big\|_X\colon \sigma_1,\dots,\sigma_k\in \{-1,1\}\bigg\}. \]
Clearly, $\|x\|_X\leqslant \seminorm{x}\leqslant K \|x\|_X$ for all $x\in X$, and so, it is enough to estimate
$\gamma(G,\seminorm{\cdot}^p)$.

To this end, observe that the space $(X,\seminorm{\cdot})$
has an $1\text{-unconditional}$ basis and cotype $q$ with constant $\mathrm{C}K$. Moreover,
since \eqref{eq-poincare} is translation invariant, it suffices to consider non-constant maps
that have $0$ as an empirical median in the sense of \eqref{eq-new-empirical}.
Thus, $\gamma(G,\seminorm{\cdot}^p)$ can be estimated by Corollary \ref{cor-strong-poincare}, the triangle
inequality and the elementary inequality $(|a|+|b|)^p \leqslant 2^{p-1} (|a|^p+|b|^p)$.
\end{proof}

\subsection{Restricted cotype and almost disjoint vectors}

In this subsection we shall prove a norm-estimate for collections of vectors in spaces with an $1\text{-unconditional}$
basis, that have restricted cotype and whose support is almost disjoint.
%We note that although this estimate is needed for the proof of Theorem~\ref{afjnobgiugbo}, it is not directly related to the main argument.
\begin{proposition}%[Restricted cotype and vector norms]
\label{aljfhbsjfhbskajhsb}
Let $k$ and $m$ be positive integers, let $\delta>0$, let $q\geqslant 2$, and let\, $\mathrm{C}\geqslant 1$.
Let $x\in\R_+^k$, and let $J_1,\dots,J_m$ be a collection of subsets of\, $[k]$ such that
\begin{equation} \label{ajfgaifgvweufygvuhgv}
\big|\big\{i\leqslant m\colon j\in J_i\big\}\big|\leqslant \delta m \ \ \ \text{ for every } j\in [k].
\end{equation}
Assume that $\|\cdot\|_X$ is a $1$-unconditional norm in $\R^k$ such that $\|P_{J_i}(x)\|_X\geqslant 1$
for every $i\in [m]$. Assume, moreover, that the collection of vectors $\big\{P_J(x)\colon J\subseteq [k]\big\}$
has restricted cotype $q$ with constant\, $\mathrm{C}$ with respect to the norm $\|\cdot\|_X$. Then, we have
\begin{equation} \label{eq-new3}
\|x\|_X^q\geqslant \frac{1}{\mathrm{C}^{q}}\cdot \frac{1}{\delta} \cdot \Big(\frac{1}{2}\Big)^{2q+5}.
\end{equation}
\end{proposition}
\begin{proof}
Clearly, we may assume that $\delta\leqslant \frac{1}{8}$. Note that if we replace any set $J_i$ with a subset $J_i' \subseteq J_i$
satisfying $\|P_{J_i'}(x)\|_X\geqslant 1$, the conditions of the proposition are still met. So, by taking subsets of each of the $J_i$
when appropriate, we may also assume that for each $i\in [m]$, we have that $\|P_{J_i}(x)\|_X\geqslant 1$ and,
either $\|P_{J_i}(x)\|_X\leqslant 2$ or $|J_i|=1$.

First, suppose that
$$ \big|\{i\leqslant m\colon |J_i|=1\}\big|\geqslant \frac{m}{2}. $$
In view of \eqref{ajfgaifgvweufygvuhgv}, there are at least $1/(2\delta)$
indices $j$ such that $\|P_{\{j\}}(x)\|_X\geqslant 1$. Hence, by the definition of restricted cotype
and the assumption that the norm $\|\cdot\|_X$ is $1$-unconditional, we obtain that
$$ \|x\|_X^q \geqslant \frac{1}{\mathrm{C}^q}\cdot \frac{1}{2\delta}, $$
that clearly implies \eqref{eq-new3}.

Next, assume that
$$ \big|\{i\leqslant m\colon |J_i|>1\}\big|\geqslant \frac{m}{2}, $$
and therefore, $2\geqslant \|P_{J_i}(x)\|_X \geqslant 1$ for at least $m/2$ indices $i\in [m]$.
Without loss of generality, this holds for all $1\leqslant i\leqslant \lceil m/2\rceil$.
Set $p:=\lceil 1/(4\delta)\rceil$, and let $i_1,i_2,\dots,i_p$ be i.i.d. uniform
random indices in $\big[\lceil m/2\rceil\big]$. We claim that for every $r \in [p]$,
$$ \Exp\,\big\| P_{J_{i_1}^\complement\cap\dots\cap J_{i_{r-1}}^\complement}\big( P_{J_{i_r}}(x)\big)\big\|_X \geqslant \frac{1}{2}, $$
where $J_i^\complement$ denotes the complement of $J_i$ in $[k]$; here, we use the convention that if $r=1$,
then $J_{i_1}^\complement\cap\dots\cap J_{i_{r-1}}^\complement = [k]$ and, consequently, this estimate is trivial.
So, fix $r\in \{2,\dots, p\}$ and notice that for every $j\in [k]$ we have
\begin{align*}
\mathbb{P}\Bigg[ j\in \bigcap_{s=1}^{r-1} J^\complement_{i_s} \Bigg]
& \ = \,\prod_{s=1}^{r-1} \mathbb{P}\big[j\notin J_{i_s}\big] =
\Bigg( \frac{\big|\big\{i\in \big[\lceil m/2\rceil\big]:\; j\notin J_i\big\}\big|}{\lceil m/2\rceil} \Bigg)^{r-1} \\
& \stackrel{\eqref{ajfgaifgvweufygvuhgv}}{\geqslant}
\Bigg( \frac{\lceil m/2\rceil - \delta m}{\lceil m/2\rceil} \Bigg)^{r-1}
\geqslant (1-2\delta)^{r-1} \geqslant 1 - 2 \delta (r-1) \geqslant \frac{1}{2}.
\end{align*}
Thus, by Jensen's inequality and the $1$-unconditionality of $\|\cdot\|_X$, we have
$$ \Exp \Big[\Exp\,\big[\big\|
P_{J_{i_1}^\complement\cap\dots\cap J_{i_{r-1}}^\complement}\big( P_{J_{i_r}}(x)\big)
\big\|_X\;\big|\;i_r\big]\Big]
\geqslant \frac{1}{2}\Exp\,\big\| P_{J_{i_r}}(x)\big\|_X\geqslant \frac{1}{2}. $$
Therefore, there is a realization of indices $i_1,\dots,i_p$ in $\big[\lceil m/2\rceil\big]$ such that
$$ \sum_{r=1}^p\big\| P_{J_{i_1}^\complement\cap\dots\cap J_{i_{r-1}}^\complement}\big( P_{J_{i_r}}(x)\big) \big\|_X \geqslant \frac{p}{2}. $$
Moreover, by the $1$-unconditionality of $\|\cdot\|_X$, we have
$\big\| P_{J_{i_1}^\complement\cap\dots\cap J_{i_{r-1}}^\complement}\big( P_{J_{i_r}}(x)\big)\big\|_X\leqslant 2$
for every $r\in [p]$. Thus, the number of indices $r\in [p]$ that satisfy
$$ \big\| P_{J_{i_1}^\complement\cap\dots\cap J_{i_{r-1}}^\complement}\big( P_{J_{i_r}}(x)\big)\big\|_X\geqslant\frac{1}{4}, $$
is at least $p/8$. Applying the definition of restricted cotype $q$ and keeping in mind that the norm $\|\cdot\|_X$ is $1$-unconditional, we obtain that $$ \|x\|_X^q \geqslant \frac{1}{\mathrm{C}^{q}}\cdot \frac{p}{8}\cdot \Big(\frac{1}{4}\Big)^q. $$ Invoking the choice of $p$, the result follows.
\end{proof}

\section{Proof of Theorem \ref*{afjnobgiugbo}} \label{sec4}

Throughout the proof, the parameters $q, \mathrm{C}, d, \alpha,\varepsilon,L$  are fixed. Let $G$ be a $d$-regular graph on $n$ vertices that satisfies property $\mathrm{Expan}(\alpha,\varepsilon,L)$, let $\|\cdot\|_X$ be a $1$-unconditional norm in $\R^k$ ($k\geqslant 1$), and fix a map $f\colon [n]\to\{-1,0,1\}^k$ such that $0$ is an empirical median of $f$ in the sense of \eqref{eq-new-empirical} and, moreover, the collection of vectors $\big\{P_J\big(|f(v)|\big)\colon v\in[n] \text{ and } J\subseteq[k]\big\}$ has restricted cotype $q$ with constant $\mathrm{C}$ with respect to the norm $\|\cdot\|_X$.

\subsection{Edges with jumps}

Our goal in this subsection is to define, for every vertex $v\in [n]$ and every coordinate $j\in [k]$ such that $f(v)_j\neq 0$,
collections of edges $\{w,w'\}$ with ``jumps'' at the $j$-th component---that is, $f(w)_j-f(w')_j\neq 0$---that
satisfy certain special properties. Those collections of edges will be employed in the next subsection to match
the contribution of $f(v)_j$, $v\in[n]$, to the sum $\sum_{v\in[n]}\|f(v)\|_X$.
The main statement in this subsection is Lemma~\ref{afjhsbfiasdjfhbdsifjhb} below.

For our definitions of subsets of edges and vertices below, we will need a series of positive numbers that sum up to one,
such that the terms of the series do not decay to zero ``too fast''\!. The particular choice of the series below is rather arbitrary.
\begin{definition}
Let $(a_i)_{i=1}^\infty$ be the sequence of positive numbers defined by
$$ a_i:=\frac{6}{\pi^2}\,\frac{1}{i^2}. $$
In particular, we have $\sum_{i=1}^\infty a_i= 1$.
\end{definition}

\begin{definition} \label{akhgfvhgfvdshgfvdjhfgv}
Let $v\in [n]$, and let $j\in [k]$. For every integer $\ell\geqslant 1$, we define
\[ E(v,j;\ell):=\big\{\{w,w'\}\in E_G\colon \dist_G(\{w,w'\},v)\leqslant \ell-1 \text{ and } f(w)_j\neq f(w')_j\big\}. \]
More generally, for every nonempty subset $S\subseteq [n]$, we set
\[ E(S,j;\ell):=\big\{\{w,w'\}\in E_G\colon \dist_G(\{w,w'\},S)\leqslant \ell-1 \text{ and } f(w)_j\neq f(w')_j\big\}. \]
\end{definition}

\begin{lemma}[Existence of large subsets $E(S,j;\ell)$]\label{afhbaofjhbojhb}
Let $j\in [k]$, and let $S$ be a nonempty set of vertices such that
$f(v)_j=f(w)_j\neq 0$ for all $v,w\in S$. Then, there exists a positive integer
$$\ell_0\leqslant \min\Big\{\ell\geqslant 1\colon \alpha(d-1)^\ell\,|S| \geqslant \frac{3n}{4}\Big\}$$
such that
$$ |E(S,j;\ell_0)|\geqslant \frac{\alpha\, a_{\ell_0}}{6}(d-1)^{\ell_0-1}\,|S|. $$
\end{lemma}
\begin{proof}
Without loss of generality, we may assume that $f(v)_j=1$ for all $v\in S$.
For every integer $\ell\geqslant 1$, set
$$ W(j;\ell):=\big\{ w\in [n]\colon \dist_G(S,w)= \ell \text{ and }  f(w)_j \neq 1 \big\}. $$
Further, set $\ell':=\min\big\{\ell\geqslant 1\colon \alpha(d-1)^\ell\,|S| \geqslant \frac{3n}{4}\big\}$.
In view of part (\hyperref[Part-A]{$A$}) of the long-range expansion property,
$$ \big|\big\{w\in[n]\colon \dist_G(w,S)\leqslant \ell'\big\}\big| \geqslant \frac{3n}{4}, $$
whereas, by the assumptions on $f$,
$$ \big|\big\{w\in[n]\colon f(w)_j\neq 1\big\}\big| \geqslant \frac{n}{2}. $$
Therefore\footnote{Note that, since $0$ is an empirical median of $f$ in the sense of \eqref{eq-new-empirical}, our assumption that $f(v)_j=f(w)_j\neq 0$ for all $v,w\in S$, yields that $|S|\leqslant \frac{n}{2}$. Thus, we may allow the somewhat degenerate case $\ell'=1$ is this estimate.},
$$ \big|\big\{w\in[n]\colon \dist_G(w,S)\leqslant \ell' \text{ and } f(w)_j \neq 1\big\}\big|
\geqslant \frac{n}{4}\geqslant \frac{\alpha}{3}(d-1)^{\ell'-1}\,|S|, $$
implying that there exists $\ell\in [\ell']$ such that
$$ |W(j;\ell)|\geqslant \frac{\alpha}{6}(d-1)^{\ell-1}\,|S|. $$
Define $\ell_0\leqslant \ell'$ to be the smallest positive integer such that
$$ |W(j;\ell_0)|\geqslant \Big(\sum_{i=1}^{\ell_0} a_i\Big)\, \frac{\alpha}{6}(d-1)^{\ell_0-1}\,|S|. $$
(By the previous argument, $\ell_0$ is well-defined.) If $\ell_0=1$, then we immediately obtain that
$|E(S,j;\ell_0)|\geqslant \frac{\alpha\,a_1}{6}(d-1)^{\ell_0-1}\,|S|$, implying the result.

Otherwise, if $\ell_0>1$, then we have
$$ |W(j;\ell_0-1)|< \Big(\sum_{i=1}^{\ell_0-1} a_i\Big)\, \frac{\alpha}{6}(d-1)^{\ell_0-2}\,|S|. $$
Moreover, each vertex $w\in W(j;\ell_0-1)$ is adjacent to at most $d-1$ vertices in $W(j;\ell_0)$. Thus,
\[ |E(S,j;\ell_0)|\geqslant |W(j;\ell_0)|-(d-1)|W(j;\ell_0-1)| \geqslant
\frac{\alpha\, a_{\ell_0}}{6}(d-1)^{\ell_0-1}\,|S|. \qedhere\]
\end{proof}

\begin{definition}
For every $v\in [n]$ and every $j\in [k]$, we define
$$ \ell_{v,j}:= \min\Big\{\ell\geqslant 1\colon |E(v,j;\ell)|\geqslant \frac{\alpha\,a_\ell}{6} (d-1)^{\ell-1}\Big\}. $$
Moreover, we set
$$ V_{+}(j;\ell):= \big\{v\in[n]\colon f(v)_j=1 \text{ and } \ell_{v,j}=\ell\big\}, $$ % \ \text{ and } \
$$ V_{-}(j;\ell):= \big\{v\in[n]\colon f(v)_j=-1 \text{ and } \ell_{v,j}=\ell\big\}. $$
\end{definition}

\begin{corollary}[Estimates for $\ell_{v,j}$, $|V_{+}(j;\ell)|$, $|V_{-}(j;\ell)|$] \label{cor3.5-new}
For every $v\in[n]$ and every $j\in [k]$ such that $f(v)_j\neq 0$, we have
$$\alpha(d-1)^{\ell_{v,j}-1}\leqslant \frac{3n}{4}.$$
Furthermore, for every integer $\ell\geqslant 1$,
\begin{equation}\label{aifhgbvfkhgwvfkhgf}
\frac{3n}{4} \geqslant \alpha(d-1)^{\ell-1}\,|V_{+}(j;\ell)|
\ \ \ \text{ and } \ \ \ \frac{3n}{4} \geqslant \alpha(d-1)^{\ell-1}\,|V_{-}(j;\ell)|.
\end{equation}
\end{corollary}
\begin{proof}
For the first assertion note that, by Lemma~\ref{afhbaofjhbojhb},
$\ell_{v,j}\leqslant \min\{\ell\geqslant 1\colon \alpha(d-1)^\ell \geqslant \frac{3n}{4}\}$.
For the second assertion, assume that $V_{+}(j;\ell)\neq\emptyset$
and observe that, by Lemma~\ref{afhbaofjhbojhb}, there exists
$p\leqslant \min\big\{\tilde\ell\geqslant 1\colon \alpha(d-1)^{\tilde\ell}\,|V_{+}(j;\ell)| \geqslant \frac{3n}{4}\big\}$ such that
$$ |E(V_{+}(j;\ell),j;p)|\geqslant \frac{\alpha\, a_{p}}{6}(d-1)^{p-1}\,|V_{+}(j;\ell)|. $$
For every edge $e\in E(V_{+}(j;\ell),j;p)$, there is a vertex $v_e\in V_{+}(j;\ell)$ with $\dist_G(e,v_e)\leqslant p-1$.
Thus, by the pigeonhole principle, there exists a vertex $v'\in V_{+}(j;\ell)$ with
$$ |E(v',j;p)| =\big|\big\{\{w,w'\}\in E_G\colon \dist_G(\{w,w'\},v')\leqslant p-1 \text{ and }
f(w)_j\neq f(w')_j\big\}\big| \geqslant \frac{\alpha\, a_{p}}{6}(d-1)^{p-1}. $$
Thus, $\ell_{v',j}\leqslant p$ and, since $v'\in V_+(j;\ell)$, we have that $\ell_{v',j}=\ell$. Hence,
$\ell\leqslant p$, which implies, by the choice of $p$, that\footnote{Recall that $0$ is an empirical median of $f$ in the sense of \eqref{eq-new-empirical}, which implies that $|V_+(j;\ell)|\leqslant \frac{n}{2}$. Hence, as in the proof of Lemma \ref{afhbaofjhbojhb}, we may also allow the somewhat degenerate case $\ell=1$.},
$$ \alpha(d-1)^{\ell-1}\,|V_{+}(j;\ell)|\leqslant \frac{3n}{4}. $$
The bound for $V_{-}(j;\ell)$ is proved similarly.
\end{proof}
Given an edge $e\in E_G$, Definition~\ref{akhgfvhgfvdshgfvdjhfgv} and Corollary \ref{cor3.5-new}
do not provide estimates for the number of vertices $v\in V_{+}(j;\ell)$ such that $e\in E(v,j;\ell)$.
Such vertex-edge configurations will be an obstacle to applying the restricted cotype property in order
to prove the main result. In the next lemma, we use the long-range expansion property of $G$ (more precisely,
part (\hyperref[Part-B]{$B$}) of Definition \ref{def-lre}) to construct large subsets of $E(v,j;\ell_{v,j})$
such that no edge of $G$ is shared by too many of those subsets.
\begin{lemma}\label{afjhsbfiasdjfhbdsifjhb}
Set
\begin{equation} \label{eq_defn_L_tilde}
\tilde{L} := 2\, \Big( \frac{20\, d\, L}{\alpha\, \varepsilon} \Big)^8.
\end{equation}
Let $j\in [k]$, let $\ell\geqslant 1$ be an integer, and let $\sigma\in \{+,-\}$ such that $V_{\sigma}(j;\ell)$ is nonempty.
Then, for every $v\in V_{\sigma}(j;\ell)$, there exists a subset $\tilde{E}(v,j)$ of $E(v,j;\ell)$
with the following properties.
\begin{itemize}
\item[(i)] For every $v\in V_{\sigma}(j;\ell)$,
\begin{equation} \label{eqkt010}
|\tilde E(v,j)|\geqslant \frac{\alpha\,a_\ell}{12}\,(d-1)^{\ell-1}.
\end{equation}
\item[(ii)] For every $e\in E_G$,
\begin{equation} \label{eqkt011}
\big|\{v\in V_{\sigma}(j;\ell)\colon e\in \tilde E(v,j)\}\big| \leqslant \tilde L(d-1-\varepsilon)^\ell.
\end{equation}
\end{itemize}
\end{lemma}
\begin{proof}
We will assume that $\sigma = +$; the case $\sigma = -$ is identical. Define $\ell_0$ to be
the least positive integer such that for any integer $\ell\geqslant\ell_0$,
\begin{equation} \label{eqkt012}
L(d-1-\varepsilon)^\ell < \frac{\alpha a_\ell}{12}(d-1)^{\ell-1}.
\end{equation}
Note that
\begin{equation}\label{eqkt013}
\ell_0\leqslant \frac{8(d-1-\varepsilon)}{\varepsilon}\, \ln\Big(\frac{20\, d\, L}{\alpha\, \varepsilon}\Big).
\end{equation}
We now fix an integer $\ell\geqslant 1$ and we consider cases.

First assume that $\ell<\ell_0$. Then, we set $\tilde E(v,j):= E(v,j;\ell)$ for every
$v\in V_{+}(j;\ell)$. With this choice, part (i) of the lemma is trivially satisfied.
Also observe that for every $e\in E_G$,
$$ \big\{v\in V_{+}(j;\ell)\colon e\in \tilde E(v,j)\big\}
\subseteq \big\{v\in [n]\colon \mathrm{dist}_G(v,e)\leqslant \ell-1\big\}, $$
and so part (ii) is also satisfied by the $d$-regularity of $G$, \eqref{eqkt013} and
the choice of $\tilde{L}$ in~\eqref{eq_defn_L_tilde}.

Next assume that $\ell\geqslant \ell_0$ and, therefore, \eqref{eqkt012} holds true. We shall
construct the desired sets $\tilde{E}(v,j)$, for each $v\in V_{+}(j;\ell)$, by an iterative greedy selection scheme.
Specifically, for every $t\in \big\{1,\ldots, |V_{+}(j;\ell)|\big\}$, we will select a vertex
$v_t\in V_{+}(j;\ell)\setminus \{v_1,\dots,v_{t-1}\}$ and
a subset $\tilde{E}(v_t,j)$ of $E(v_t,j;\ell)$ such that, setting
\[ \mathcal{G}_t:=
\Big(\bigcup_{s=1}^t \{v_s\}\times \tilde{E}(v_s,j)\Big)
\cup \Big(\bigcup_{v\in V_{+}(j;\ell)\setminus\{v_1,\ldots,v_t\}} \{v\}\times E(v,j;\ell)\Big),\]
the following properties are satisfied.
\begin{enumerate}
\item [(P1)] We have $|\tilde{E}(v_t,j)|\geqslant \frac{\alpha\,a_\ell}{12}\,(d-1)^{\ell-1}$.
\item [(P2)] For every $e\in \tilde{E}(v_t,j)$, we have
$\big|\{v\in V_{+}(j;\ell)\setminus\{v_1,\ldots,v_{t-1}\}\colon (v,e)\in
\mathcal{G}_t\}\big| \leqslant  L(d-1-\varepsilon)^\ell$.
\end{enumerate}
Let $t\in \big[|V_{+}(j;\ell)|\big]$ and assume that the construction has been carried out up to $t-1$.
Define
\[ E_t:=\big\{e\in E_G\colon \big|\{v\in V_{+}(j;\ell)\setminus\{v_1,\dots,v_{t-1}\}\colon
e\in E(v,j;\ell)\}\big|\geqslant L(d-1-\varepsilon)^\ell\big\}, \]
\[ T_t:= \Big\{e\in E_G\colon
\big|\{v\in V_{+}(j;\ell)\setminus\{v_1,\dots,v_{t-1}\}\colon
\dist_G(v,e)\leqslant \ell-1\}\big|\geqslant L(d-1-\varepsilon)^\ell\Big\}.\]
Clearly, $E_t$ is a subset of  $T_t$
and, in view of \eqref{aifhgbvfkhgwvfkhgf},
$$ \alpha(d-1)^{\ell-1}\, \big|V_{+}(j;\ell)\setminus\{v_1,\dots,v_{t-1}\}\big| \leqslant \frac{3n}{4}. $$
Hence, by part (\hyperref[Part-B]{$B$}) of the long-range expansion property applied for
``$S=V_{+}(j;\ell)\setminus\{v_1,\dots,v_{t-1}\}$", there exists a vertex
$v_t\in V_{+}(j;\ell)\setminus\{v_1,\dots,v_{t-1}\}$ such that
$$ \big|\{e\in E_G\colon \dist_G(e,v_t)\leqslant \ell-1\} \cap E_t\big|\leqslant L(d-1-\varepsilon)^\ell, $$
implying that
$$ |E(v_t,j;\ell)\setminus E_t| \geqslant \frac{\alpha\,a_\ell}{12}\,(d-1)^{\ell-1}. $$
We set
$$ \tilde E(v_t,j):=E(v_t,j;\ell)\setminus E_t. $$
Clearly, properties (P1) and (P2) of the selection are satisfied.

We claim that the sets $\tilde{E}(v,j)$, $v\in V_{+}(j;\ell)$, satisfy the requirements of the lemma.
Indeed, notice first that part (i) follows from property (P1) of the selection. To see that part (ii)
is also satisfied, fix $e\in E_G$. If $e\notin \bigcup_{v\in V_+(j,\ell)} \tilde E(v,j)$,
then the set $\{v\in V_{\sigma}(j;\ell):\;e\in \tilde E(v,j)\}$ is empty; hence, \eqref{eqkt011}
is trivially satisfied. So assume that $e\in \bigcup_{v\in V_+(j;\ell)} \tilde E(v,j)$, and set
$$\tau:= \min\big\{ t\in \big[|V_+(j;\ell)|\big]\colon e \in \tilde{E}(v_t,j)\big\}.$$
Notice that $\mathcal{G}_{|V_+(j;\ell)|}\subseteq \mathcal{G}_\tau$ and therefore, by the
definitions of $\mathcal{G}_{|V_+(j;\ell)|}$ and $\tau$,
\[\begin{split}
\big\{v\in V_+(j;\ell)\colon & e\in \tilde E(v,j)\big\}
= \big\{v\in V_+(j;\ell)\colon (v,e)\in \mathcal{G}_{|V_+(j;\ell)|}\big\}\\
& \subseteq \big\{v\in V_+(j;\ell)\colon (v,e)\in \mathcal{G}_\tau \big\}
= \big\{v\in V_+(j;\ell)\setminus \{v_1,\ldots,v_{\tau-1}\}\colon (v,e)\in \mathcal{G}_\tau\big\}.
\end{split}\]
Hence, by property (P2) of the selection,
\[ \big|\{v\in V_+(j;\ell)\colon e\in \tilde E(v,j)\}\big| \leqslant L(d-1-\varepsilon)^\ell.\]
Since $L\leqslant\tilde{L}$, the result follows.
\end{proof}

\subsection{Recombination}

\begin{definition} \label{def-jvl}
Given a vertex $v\in[n]$ and an integer $\ell\geqslant 1$, let $J(v,\ell)$ denote the collection of all indices $j\in [k]$
such that $v\in V_+(j;\ell)\cup V_-(j;\ell)$; moreover, for every $e\in E_G$, we set
$$ J(v,\ell,e):=\big\{j\in J(v,\ell)\colon e\in \tilde E(v,j)\big\}, $$
where the sets of edges $\tilde{E}(v,j)$ are taken from Lemma~\ref{afjhsbfiasdjfhbdsifjhb}.
\end{definition}
\begin{remark} \label{rem-new-3.8}
Observe that, by Lemma~\ref{afjhsbfiasdjfhbdsifjhb},
for every $j\in [k]$, every $\ell\geqslant 1$, and every $e\in E_G$,
$$ \big|\{v\in[n]\colon j\in J(v,\ell,e)\}\big|\leqslant 2\tilde L(d-1-\varepsilon)^\ell. $$
Also notice that, by the unconditionality of the norm $\|\cdot\|_X$, for every vertex $v\in [n]$ we have that
$\|\chi_{J(v,\ell)}\|_X=\big\|P_{J(v,\ell)}\big(f(v)\big)\big\|_X$.
\end{remark}

\begin{lemma} \label{aifgvsifghvfihfgiw}
Let $b\in\Z$, let $\ell\geqslant 1$, and let $v\in [n]$ such that
$\|\chi_{J(v,\ell)}\|_X=\big\|P_{J(v,\ell)}\big(f(v)\big)\big\|_X\geqslant 2^b$.
Then there exist at least $\mathrm{c}\,a_\ell (d-1)^\ell$ edges $e\in E_G$ such that
$\|\chi_{J(v,\ell,e)}\|_X\geqslant \mathrm{c}\,a_\ell\,2^b$, where
\begin{equation} \label{eq-c1}
\mathrm{c}:= \frac{\alpha^2}{48d(d-1)}.
\end{equation}
\end{lemma}
\begin{proof}
We start by setting
$$ \tilde E:=\big\{e\in E_G\colon \dist_G(v,e)\leqslant \ell-1\big\}, $$
and we claim that
\begin{equation} \label{eqkt014}
\alpha (d-1)^{\ell-1} \leqslant  |\tilde E|\leqslant 2d(d-1)^{\ell-1}.
\end{equation}
The upper bound in \eqref{eqkt014} follows readily from the $d$-regularity of $G$. To see that the lower bound
in \eqref{eqkt014} is also satisfied, notice first that
\begin{equation} \label{eqkt015}
|\tilde E| \geqslant \big|\{w\in [n]\colon \dist_G(v,w) \leqslant \ell\}\big|-1 \geqslant
\big|\{w\in [n]\colon \dist_G(v,w) \leqslant \ell-1\}\big|.
\end{equation}
(The second inequality in \eqref{eqkt015} follows from the connectivity of $G$, which is a
consequence of part~(\hyperref[Part-A]{$A$}) of the long-range expansion property.)
On the other hand, since $\|\chi_{J(v,\ell)}\|_X>0$, the set $J(v,\ell)$ is nonempty. Pick any
$j\in J(v,\ell)$. By the definition of $J(v,\ell)$, we see that $\ell = \ell_{v,j}$ and consequently,
by \eqref{aifhgbvfkhgwvfkhgf}, we obtain that $\alpha(d-1)^{\ell-1}\leqslant \frac{3n}{4}$.
By part (\hyperref[Part-A]{$A$}) of the long-range expansion property applied for ``$S = \{v\}$", we get that
\[ \big|\{w\in [n]\colon \dist_G(v,w) \leqslant \ell-1\}\big|\geqslant \alpha(d-1)^{\ell-1}.\]
The desired lower bound in \eqref{eqkt014} follows from this estimate together with \eqref{eqkt015}.

Let $\xi$ be uniform random edge in $\tilde E$. By \eqref{eqkt010}, \eqref{eqkt014} and the fact that
$\tilde{E}(v,j) \subseteq \tilde{E}$, we have
$$ \Prob\big[j\in J(v,\ell,\xi)\big]=\frac{|\tilde E(v,j)|}{|\tilde E|}
\geqslant \frac{\alpha\,a_\ell}{24d} \ \ \ \text{ for every } j\in J(v,\ell). $$
Thus, by the convexity of the norm and Jensen's inequality,
$$ \Exp\,\|\chi_{J(v,\ell,\xi)}\|_X\geqslant \frac{\alpha\,a_\ell\,\|\chi_{J(v,\ell)}\|_X}{24d}. $$
On the other hand, everywhere on the probability space we have
$$ \|\chi_{J(v,\ell,\xi)}\|_X\leqslant \|\chi_{J(v,\ell)}\|_X. $$
Applying the reverse Markov inequality, we obtain
\begin{align*}
|\tilde E| \, \frac{\alpha\,a_\ell\, \|\chi_{J(v,\ell)}\|_X }{24d} \leqslant
\big|\big\{\xi \in \tilde E\colon \|\chi_{J(v,\ell,\xi)}\|_X \ge c\,a_\ell \,2^b \big\}\big| \cdot
\|\chi_{J(v,\ell)}\|_X + (c\,a_\ell\, 2^b)\,|\tilde E|.
\end{align*}
The result follows after using the lower bound~in~\eqref{eqkt014} on $|\tilde E|$, and recalling the assumption that
$\|\chi_{J(v,\ell)}\|_X \ge 2^b$.
\end{proof}

\begin{lemma}\label{afokjnlfkjnsadfkjhsfb}
Let $b\in\Z$, let $\ell\geqslant 1$ be an integer, and set
$$ M(b,\ell):= \big\{v\in[n]\colon \big\|P_{J(v,\ell)}\big(f(v)\big)\big\|_X\geqslant 2^b\big\}. $$
Also set
\begin{equation} \label{eq-c-hat}
\hat{\mathrm{c}} := \frac{\alpha^3}{2^{15}\, d^3\, \mathrm{C}\, \tilde{L}^{1/q}},
\end{equation}
where $\tilde{L}$ is as in \eqref{eq_defn_L_tilde}. Then, either
\begin{itemize}
\item[(i)] the number of edges $\{w,w'\}\in E_G$ with
$$\|f(w)-f(w')\|_X\geqslant \hat{\mathrm{c}}\,a_\ell\,2^b$$
is at least $\frac{1}{a_\ell^2}\,|M(b,\ell)|$, or
\item[(ii)] the number of edges $\{w,w'\}\in E_G$ with
$$ \|f(w)-f(w')\|_X\geqslant \hat{\mathrm{c}}\,a_\ell^3\,2^b\,\bigg(\frac{d-1}{d-1-\varepsilon}\bigg)^{\ell/q} $$
is at least $\hat{\mathrm{c}}\,a_\ell\,|M(b,\ell)|$.
\end{itemize}
\end{lemma}
\begin{proof}
For every $v\in M(b,\ell)$, let $E_v$ denote the set of all edges $e\in E_G$ with
$\|\chi_{J(v,\ell,e)}\|_X\geqslant \mathrm{c}\,a_\ell\,2^b$, where the constant $\mathrm{c}$ is as in \eqref{eq-c1}.
By Lemma~\ref{aifgvsifghvfihfgiw},
\begin{equation} \label{akhgfvjfhgvfjhgv}
|E_v|\geqslant \mathrm{c}\,a_\ell\,(d-1)^\ell \ \ \ \text{ for every } v\in M(b,\ell).
\end{equation}
Further, set
$$ E':=\bigcup\limits_{v\in M(b,\ell)}E_v. $$
For every $\{w,w'\}\in E_v$ and every $j\in J(v,\ell,\{w,w'\})$, we have that
$\{w,w'\}\in \tilde E(v;j)\subseteq E(v,j;\ell_{v,j})$, which implies that $f(w)_j-f(w')_j\neq 0$. Hence,
$$ \supp\big(f(w)-f(w')\big)\supseteq J(v,\ell,\{w,w'\}), $$
which in turn implies, by the unconditionality of $\|\cdot\|_X$, that
$$ \|f(w)-f(w')\|_X\geqslant \mathrm{c}\,a_\ell\,2^b. $$
Therefore, if $|E'|\geqslant \frac{1}{a_\ell^2}\,|M(b,\ell)|$, then, after observing that $\mathrm{c}\geqslant \hat{\mathrm{c}}$,
we see that part (i) of the lemma is satisfied.

So, assume that $|E'|<\frac{1}{a_\ell^2}\,|M(b,\ell)|$. Note that, by $d$-regularity of $G$,
for every $e\in E_G$ there are at most $\frac{2}{d-2}\,(d-1)^\ell$ distinct vertices $v\in[n]$ with
$\dist_G(e,v)\leqslant \ell-1$. Consequently, by the choice of~$E_v$, the definitions of $J(v,\ell,e)$ and
$E(v,j;\ell)$ in Definitions \ref{def-jvl} and \ref{akhgfvhgfvdshgfvdjhfgv} respectively,
and Lemma~\ref{afjhsbfiasdjfhbdsifjhb}, we obtain that
$$ \big|\{v\in M(b,\ell)\colon e\in E_v\}\big|\leqslant \frac{2}{d-2}\,(d-1)^\ell \ \ \ \text{ for every } e\in E_G. $$
Combining this with \eqref{akhgfvjfhgvfjhgv}, we obtain
\begin{align*}
\mathrm{c}\,a_\ell\,(d-1)^\ell\,|M(b,\ell)| &\leqslant  \frac{2}{d-2}(d-1)^\ell \,\left|\left\{e \in E'\colon
\big|\{v\in M(b,\ell)\colon e\in E_v\}\big|\geqslant \frac{\mathrm{c}}{4}\,a_\ell^3\,(d-1)^\ell\right\}\right|  \\
& \ \ \ \ \ \ \  + |E'|\, \frac{c}{4}a_\ell^3\,(d-1)^\ell.
\end{align*}
Rearranging and using the assumption on the cardinality of $E'$, we have that
$$ \big|\{v\in M(b,\ell)\colon e\in E_v\}\big|\geqslant \frac{\mathrm{c}}{4}\,a_\ell^3\,(d-1)^\ell $$
for at least $\frac{\mathrm{c}}{4}\,a_\ell\,|M(b,\ell)|$ edges $e\in E'$.
Fix any such edge $e=\{w,w'\}$. Observe that
$$ \supp\big(f(w)-f(w')\big)\supseteq J(v,\ell,\{w,w'\}) \ \ \ \text{ whenever } \{w,w'\}\in E_v; $$
moreover, by Remark \ref{rem-new-3.8}, for every $j\in [k]$, we have
$$ \big|\big\{ v\in M(b,\ell)\colon j\in J(v,\ell,\{w,w'\}\big\}\big| \leqslant 2\tilde L(d-1-\varepsilon)^\ell. $$
By Proposition~\ref{aljfhbsjfhbskajhsb}, we obtain that
$$ \|f(w)-f(w')\|_X \geqslant \hat{\mathrm{c}}\,a_\ell^3\,2^b\,\bigg(\frac{d-1}{d-1-\varepsilon} \bigg)^{\ell/q}. $$
Since $\frac{\mathrm{c}}{4}\geqslant \hat{\mathrm{c}}$, this shows that part (ii) of the lemma is satisfied.
\end{proof}

\begin{proof}[Proof of Theorem~\ref{afjnobgiugbo}]
Fix any integer $\ell\geqslant 1$. Define two disjoint, possibly empty, subsets $B_1(\ell)$ and $B_2(\ell)$ of $\mathbb{Z}$
as follows. If $b\in \mathbb{Z}$ is such that $M(b,\ell)$ is nonempty,
then include $b$ in $B_1(\ell)$ if part (i) of Lemma \ref{afokjnlfkjnsadfkjhsfb} is satisfied; otherwise,
include $b$ in $B_2(\ell)$. In particular, if $b\in B_2(\ell)$, then
$M(b,\ell)\neq\emptyset$ and part (ii)
of Lemma \ref{afokjnlfkjnsadfkjhsfb} is satisfied.

Next observe that for every $\gamma>0$,
\begin{equation}\label{eqkt017}
2 \sum_{\{w,w'\}\in E_G}\|f(w)-f(w')\|_X \geqslant
\sum_{b\in\mathbb{Z}}\gamma\, 2^b\, \big|\big\{ \{w,w'\}\in E_G\colon \|f(w)-f(w')\|_X\geqslant \gamma 2^b\big\}\big|.
\end{equation}
By the definition of $B_1(\ell)$, part (i) of Lemma \ref{afokjnlfkjnsadfkjhsfb} and \eqref{eqkt017} applied for
``$\gamma= \hat{\mathrm{c}}\,a_\ell$'', we get that
\begin{align*}
2\sum_{\{w,w'\}\in E_G}\|f(w)-f(w')\|_X &\geqslant \sum\limits_{b\in B_1(\ell)} \hat{\mathrm{c}}\,a_\ell\,2^b\,
\big|\big\{\{w,w'\}\in E_G\colon \|f(w)-f(w')\|_X\geqslant \hat{\mathrm{c}}\,a_\ell\,2^b\big\}\big|\\
&\geqslant \sum\limits_{b\in B_1(\ell)}\hat{\mathrm{c}}\,a_\ell\,2^b\cdot
\frac{1}{a_\ell^2}|M(b,\ell)|\\
&= \frac{\hat{\mathrm{c}}}{a_\ell}\sum\limits_{b\in B_1(\ell)}2^b\,
\big|\big\{v\in[n]\colon \big\|P_{J(v,\ell)}\big(f(v)\big)\big\|_X\geqslant 2^b\big\} \big|.
\end{align*}
Similarly, by the definition of $B_2(\ell)$, part (ii) of Lemma \ref{afokjnlfkjnsadfkjhsfb} and inequality \eqref{eqkt017}
applied this time for ``$\gamma= \hat{\mathrm{c}}\,a_\ell^3 \big(\frac{d-1}{d-1-\varepsilon}\big)^{\ell/q}$'',
\begin{align*}
2&\sum_{\{w,w'\}\in E_G}\|f(w)-f(w')\|_X\\
&\geqslant \hat{\mathrm{c}}\,a_\ell^3\,
\sum\limits_{b\in B_2(\ell)}
2^b\,\bigg(\frac{d-1}{d-1-\varepsilon}\bigg)^{\ell/q}\
\bigg|\Big\{\{w,w'\}\in E_G\colon \|f(w)-f(w')\|_X\geqslant
\hat{\mathrm{c}}\,a_\ell^3\,2^b\,\bigg(\frac{d-1}{d-1-\varepsilon}\bigg)^{\ell/q}
\Big\}\bigg|
\\
&\geqslant\hat{\mathrm{c}}\,a_\ell^3\,
\sum\limits_{b\in B_2(\ell)}
2^b\,\bigg(\frac{d-1}{d-1-\varepsilon}\bigg)^{\ell/q}\,
\hat{\mathrm{c}}\,a_\ell\,|M(b,\ell)|\\
&=
\hat{\mathrm{c}}^2\,a_\ell^4\,\bigg(\frac{d-1}{d-1-\varepsilon}\bigg)^{\ell/q}
\sum\limits_{b\in B_2(\ell)}
2^b\,\big|\big\{v\in[n]\colon \big\|P_{J(v,\ell)}\big(f(v)\big)\big\|_X\geqslant 2^b\big\}
\big|.
\end{align*}
Set $\mathrm{c}':= \hat{\mathrm{c}}^2\, \big(\frac{\varepsilon}{10qd}\big)^{10}$ and observe that
$\mathrm{c}'\leqslant \hat{\mathrm{c}}$ and $\frac{4}{\mathrm{c}'}=\Pi$, where $\Pi$ is as in \eqref{eq-new2}.
Moreover, by the definition of~$a_\ell$, we have
$$ \hat{\mathrm{c}}^2\,a_\ell^4\,\bigg(\frac{d-1}{d-1-\varepsilon}\bigg)^{\ell/q}
\geqslant \frac{\mathrm{c}'}{a_\ell}. $$
Combining the above estimates, we obtain that
\begin{align*}
4\sum_{\{w,w'\}\in E_G}\|f(w)-f(w')\|_X
&\geqslant \frac{\mathrm{c}'}{a_\ell}
\sum\limits_{b\in \Z} 2^b\,\big|\big\{v\in[n]\colon \big\|P_{J(v,\ell)}\big(f(v)\big)\big\|_X \geqslant 2^b\big\} \big|\\
&\geqslant \frac{\mathrm{c}'}{a_\ell} \sum\limits_{v\in[n]} \big\|P_{J(v,\ell)}\big(f(v)\big)\big\|_X.
\end{align*}
Finally, observe that $\supp\big(f(v)\big)=\bigcup_{\ell\geqslant 1} J(v,\ell)$ for every $v\in [n]$.
Thus, summing over all $\ell\geqslant 1$ and using the triangle inequality and the fact that $\sum_{\ell=1}^\infty a_\ell=1$,
we conclude that
\[ \sum\limits_{v\in[n]} \|f(v)\|_X\leqslant
\sum\limits_{\ell\geqslant 1} \sum\limits_{v\in[n]} \big\|P_{J(v,\ell)}\big(f(v)\big)\big\|_X
\leqslant \frac{4}{\mathrm{c}'}\sum_{\{w,w'\}\in E_G}\|f(w)-f(w')\|_X. \qedhere\]
\end{proof}

\section{Long-range expansion is typical: proof of Proposition \ref*{lemma:long-range-expansion}} \label{sec5}

Our goal in this section is to present the proof of Proposition \ref{lemma:long-range-expansion}.

Part (a) is proven by a direct computation in the configuration model. Given a (not too large) set
$S$ of vertices, we analyze the exploration process associated with running a breadth-first search from~$S$.
The key quantity is $\Delta(S,\ell)$, the number of previously unexplored vertices revealed at step~$\ell$.
Using an inductive argument, in Proposition \ref{p3-pap1} we obtain strong lower bounds on the growth of $\Delta(S,\ell)$,
until a (small) constant fraction of all vertices have been discovered. At this point, it is sufficient to
crudely bound the remainder of the exploration process using Cheeger's inequality.
A main technical difficulty in this inductive strategy is actually the base case: if $d$ is small---in particular, $d= 3$ is the core issue---then $|\Delta(S,\ell)|$ is poorly concentrated for small values of $\ell$. This is resolved by invoking Cheeger's inequality in Fact \ref{fact_new_Cheeger_expansion} and Corollary \ref{cor_new_base_step} to obtain weak lower bounds on $|\Delta(S,\ell)|$ for small $\ell$. Once $\ell$ is sufficiently large, the exploration process has encountered sufficiently many total vertices for $|\Delta(S,\ell)|$ to be strongly concentrated.

For proving part (b) of Proposition \ref{lemma:long-range-expansion}, we utilize a spectral argument. Notice that part (\hyperref[Part-B]{$B$}) of the long-range
expansion property roughly states that for a large set $S$, it is not possible for a disproportionate fraction
of the walks of length $\ell$ started from $S$ to all end up in a small set $T$. Let $G$ be a graph as in
part (b), and let $A_G$ be its adjacency matrix; then the spectral gap of $A_G^{\ell}$ can be lower bounded
in terms of the spectral gap of $A_G$. And, assuming the negation of part (\hyperref[Part-B]{$B$}),
we explicitly construct a vector which certifies a small spectral gap for $A_G^{\ell}$, yielding a contradiction.

Before we proceed to the details, we first present some basic facts concerning random regular graphs.

\subsection*{The configuration model}

This model, introduced by Bollob\'{a}s \cite{Bol80}, is also referred to as the \emph{pairing model}. We briefly recall its definition and a couple of its basic properties that are needed for the proof of part (a) of Proposition \ref{lemma:long-range-expansion}; for further information we refer to \cite{Bol01,JLR00,Wo99}.

Fix two integers $n> d\geqslant 3$ such that $dn$ is even, and consider the set $[n]\times [d]$; we view this set as starting with $n$ vertices and drawing $d$ lines below each of them. Given a uniformly random perfect matching $\mu$ on $[n]\times [d]$, let $\mathbf{G}=\mathbf{G}(\mu)$ be the random multigraph obtained by collapsing the $d$ points below a vertex $v\in [n]$ to the vertex $v$.

Let $\mathcal{S}$ denote the event that the multigraph $\mathbf{G}$ is simple, that is, it has no multiple edges and no self-loops. It is easy to see that the distribution of $\mathbf{G}$ conditioned on $\mathcal{S}$ coincides with that of a random $d$-regular graph uniformly sampled from $G(n,d)$. More importantly, we have
\begin{equation} \label{e-ra-s-1}
\mathbb{P}[\mathcal{S}] = \big(1+o_d(1)\big)\, e^{-(d^2-1)/4};
\end{equation}
in particular, if a property is satisfied with high probability in the configuration model, then it is also satisfied with high probability in $G(n,d)$.

\subsection*{Friedman's second eigenvalue theorem}

Let $n > d\geqslant 3$ be integers, and let $G$ be a $d$-regular graph on $[n]$. Recall that by $\lambda_n(G)\leqslant \cdots \leqslant \lambda_2(G) \leqslant \lambda_1(G)$ we denote the eigenvalues of the adjacency matrix of $G$, and we set
$\lambda(G)=\max\big\{ |\lambda_2(G)|,\dots,|\lambda_n(G)\big\}$. We will need the following result of Friedman \cite{Fr08}, formerly known as Alon's conjecture.

\begin{theorem} \label{Friedman}
Let $n>d\geqslant 3$ be integers such that $dn$ is even, and let $\mathbb{P}$ be the uniform probability measure on $G(n,d)$. There exists a constant $\tau=\tau(d)>0$ such that for every $\epsilon>0$,
\[ \mathbb{P}\big[ \lambda(G) \leqslant 2\sqrt{d-1} + \epsilon \big] \geqslant
1- O_{d,\epsilon}\Big( \frac{1}{n^{\tau}}\Big). \]
\end{theorem}

\subsection*{The Cheeger constant of regular graphs}

Let $n>d\geqslant 3$ be integers, and let $G$ be a $d$-regular graph on $[n]$; the \emph{Cheeger constant} of $G$---also known as
\emph{expansion ratio}---is defined by
\[ h(G):= \min_{\substack{\emptyset \neq S\subseteq [n]\\ |S|\leqslant \frac{n}{2}}}
\frac{\big|\big\{ \{v,w\}\in E_G:\; v\in S \text{ and } w\in S^{\complement}\big\}\big|}{|S|}. \]
It is well-known that the Cheeger constant of $G$ is related to the spectral properties of the
adjacency matrix of $G$. More precisely, we have
\begin{equation} \label{e-cheeger-e1}
\frac{d-\lambda_2(G)}{2} \leqslant h(G) \leqslant \sqrt{2d\big(d-\lambda_2(G)\big)};
\end{equation}
see, e.g., \cite[Theorem 4.1]{HLW06}.
%We will need the following estimate of the Cheeger constant of random $d$-regular graphs.
%\begin{corollary} \label{l-pap-friedman}
%Let $n\geqslant d\geqslant 3$ be integers such that $nd$ is even, and let $\mathbb{P}$ be the uniform probability measure on $G(n,d)$.
%There exists a constant $\tau=\tau(d)>0$ such that
%\[ \mathbb{P}\big[h(G)\geqslant 0.0016\, d\big] \geqslant 1- O_d \Big(\frac{1}{n^{\tau}}\Big). \]
%\end{corollary}
%\begin{proof}
%It follows from Theorem \ref{Friedman} and \eqref{e-cheeger-e1}.
%\end{proof}

\subsection{Proof of part (a)}

For the rest of this proof we fix a pair $n>d\geqslant 3$ of integers such that $dn$ is even. It will also be convenient to introduce the following notation. For every multigraph $G$ on $[n]$, every $S\subseteq [n]$ and every integer $\ell$, we set
\begin{enumerate}
\item[$\bullet$] $B_G(S,\ell) := \big\{ v\in[n]\colon \mathrm{dist}_G (v,S)\leqslant\ell\big\}$,
\item[$\bullet$] $\partial B_G(S,\ell):=  B_G(S,\ell) \setminus B_G(S,\ell-1)$,
\item[$\bullet$] $\Delta_G(S,\ell) := \big\{v\in\partial B_G(S,\ell)\colon \text{ there is a unique edge connecting $v$ to } B_G(S,\ell-1)\big\}$,
\end{enumerate}
with the convention that $B_G(S,\ell)=\emptyset$ if $S$ is empty or if $\ell$ is negative.

\begin{definition}[Perfect matchings, partial matchings, and related objects]
Let
\begin{align}
\mathcal{M} & := \big\{\mu\colon \text{$\mu$ is a perfect matching of\, } [n]\times[d]\big\}, \nonumber \\
\mathcal{M}_\mathrm{part} & := \big\{ \pi\colon \text{$\pi$ is a (possibly partial) matching of\, }[n]\times[d]\big\}; \nonumber
\end{align}
notice that $\mathcal{M}\subseteq \mathcal{M}_\mathrm{part}$. For every $\pi\in\mathcal{M}_\mathrm{part}$, we set
\begin{align*}
\cup\pi &:= \big\{(v,k)\in [n]\times [d]\colon (v,k)\in e \text{ for some } e\in\pi\big\}, &
\mathcal{M}_{\pi} & := \big\{ \mu\in\mathcal{M}\colon \pi \subseteq \mu\big\}, \\
V(\pi) & := \big \{v\in [n]\colon \big(\{v\}\times [d]\big) \cap (\cup\pi)\neq\emptyset\big\}, &
\mathrm{Free}(\pi) & := \big(V(\pi) \times [d]\big) \setminus \cup\pi;
\end{align*}
moreover, by $G(\pi)$ we denote the multigraph on $[n]$ where we add an edge connecting $v,w\in [n]$
whenever there exist $k,l\in [d]$ such that $\big\{(v,k),(w,l)\big\} \in \pi$.
Observe that if $\pi\in\mathcal{M}$---that is, if $\pi$ is a perfect matching---then $G(\pi)$ is the multigraph
$\mathbf{G}(\pi)$ associated with $\pi$ in the configuration model.
\end{definition}
\begin{lemma} \label{l1-pap}
Let $\pi\in\mathcal{M}_\mathrm{part}$, let $0<\theta<1$, let\, $R\subseteq [n]$ with $V(\pi)\subseteq R$
and $|R|<\frac{n}{2}$, and set $\mathcal{A} := \big(R\times [d]\big)\setminus \cup\pi$. Then,
\begin{equation} \label{e-pap-1}
\mathbb{P}\Big[ \big|\Delta_{\mathbf{G}}(R,1)\big|\geqslant
\theta|\mathcal{A}| \, \Big| \, \mathcal{M}_{\pi} \Big]
\geqslant 1- \Big(\frac{2 e}{1-\theta} \cdot \frac{|\mathcal{A}|}{n-2|R|}
\Big)^{\frac{1-\theta}{2}|\mathcal{A}|}.
\end{equation}
\end{lemma}
\begin{proof}
Set $m:=|\mathcal{A}|$. Introduce the random variables $\mathbf{X}, \mathbf{X}_{\mathrm{in}}$
and $\mathbf{X}_{\mathrm{out}}$ on $\mathcal{M}_{\pi}$ by the rule
\[ \mathbf{X}(\mu) := \big\{ e\in\mu\setminus\pi\colon e\cap \mathcal{A}\neq\emptyset \big\}, \ \ \
\mathbf{X}_{\mathrm{in}}(\mu):=  \big\{ e\in\mu\setminus \pi\colon e \subseteq \mathcal{A}\big\},
\ \ \ \mathbf{X}_{\mathrm{out}}(\mu) := \mathbf{X}(\mu) \setminus \mathbf{X}_{\mathrm{in}}(\mu). \]
Notice that everywhere on $\mathcal{M}_\pi$ we have
\[ \frac{m}{2}\leqslant |\mathbf{X}| \leqslant m; \]
moreover, $|\mathbf{X}_{\mathrm{in}}|+|\mathbf{X}_{\mathrm{out}}|=|\mathbf{X}|$
and $2 |\mathbf{X}_{\mathrm{in}}|+ |\mathbf{X}_{\mathrm{out}}|= m$, which implies that
\begin{equation} \label{eqkt19}
|\mathbf{X}_{\mathrm{in}}|= m- |\mathbf{X}| \ \ \ \text{ and } \ \ \
|\mathbf{X}_{\mathrm{out}}| = 2|\mathbf{X}| - m.
\end{equation}
Next observe that the probability on the left-hand side of \eqref{e-pap-1} depends only on the $\mathbf{X}(\mu)$ part
of~$\mu\in\mathcal{M}_\pi$. We will reveal this part sequentially. Specifically, fix an arbitrary linear order $\prec$
on~$\mathcal{A}$. Then, given $\mu\in\mathcal{M}_{\pi}$, we enumerate the set $\mathbf{X}(\mu)$ by
$\mathbf{e}_1(\mu),\dots,\mathbf{e}_{|\mathbf{X}(\mu)|}(\mu)$ so that for every $i,j\in \big\{1,\dots, |\mathbf{X}(\mu)|\big\}$
with $i<j$ we have that\footnote{Here, the minimum is taken with respect to $\prec$.}
$\min\big(\mathbf{e}_i(\mu)\big) \prec \min\big(\mathbf{e}_j(\mu)\big)$.

For every $i\in [m]$, define $\mathcal{E}_i$ as the event that the $i$-th edge of $\mathbf{X}$ connects to a vertex outside $R$
that no other edge from $\mathbf{X}$ is connected to; namely,
\begin{align*}
\mathcal{E}_i := \Big\{ \mu\in \mathcal{M}_{\pi} \colon & i\leqslant |\mathbf{X}(\mu)| \text{ and } \big(\{v\}\times[d]\big) \cap \mathbf{e}_j(\mu) = \emptyset
\text{ for all } j\in \big\{1,\dots,|\mathbf{X}(\mu)|\big\} \text{ with } j\neq i, \\
& \text{where } \big\{(v,k)\big\} := \mathbf{e}_i(\mu)\setminus \big(R\times [d]\big) \Big\}.
\end{align*}
Similarly, we also define $\mathcal{F}_i$ as the event that the $i$-th edge of $\mathbf{X}$ connects
to a new vertex outside $R$ at the time it is revealed,
\begin{align*}
\mathcal{F}_i := \Big\{ \mu\in \mathcal{M}_{\pi}\colon & i\leqslant |\mathbf{X}(\mu)| \text{ and } \big(\{v\}\times[d]\big) \cap \mathbf{e}_j(\mu) = \emptyset
\text{ for all } j < i, \\
& \text{where } \big\{(v,k)\big\} := \mathbf{e}_i(\mu)\setminus \big(R\times [d]\big) \Big\}.
\end{align*}
Notice~that
\[ \big|\Delta_{\mathbf{G}}(R,1)\big| = \sum_{i=1}^{|\mathbf{X}|} \mathbbm{1}_{\mathcal{E}_i}. \]
On the other hand, everywhere on $\mathcal{M}_\pi$, for every $i\in \big\{1,\dots,|\mathbf{X}|\big\}$, we have
\[ \mathbbm{1}_{\mathcal{E}_i}=1 \Rightarrow \mathbbm{1}_{\mathcal{F}_i} =1 \Rightarrow
\mathbf{e}_i\in \mathbf{X}_{\mathrm{out}} \ \ \ \text{ and } \ \ \
\mathbf{e}_i\in \mathbf{X}_{\mathrm{in}} \Rightarrow \mathbbm{1}_{\mathcal{F}_i^\complement} =1
\Rightarrow \mathbbm{1}_{\mathcal{E}_i^\complement} =1. \]
Moreover,
\begin{equation} \label{malakia}
\sum_{\substack{i\in \{1,\dots,|\mathbf{X}|\}\\\mathbf{e}_i\in \mathbf{X}_\mathrm{out}}} \mathbbm{1}_{\mathcal{F}_i^\complement} \geqslant
\sum_{\substack{i\in \{1,\dots,|\mathbf{X}|\}\\ \mathbf{e}_i\in \mathbf{X}_\mathrm{out}}} \mathbbm{1}_{\mathcal{F}_i} -
\sum_{\substack{i\in \{1,\dots,|\mathbf{X}|\}\\ \mathbf{e}_i\in \mathbf{X}_\mathrm{out}}} \mathbbm{1}_{\mathcal{E}_i};
\end{equation}
indeed, first observe that for every $i\leqslant |\mathbf{X}|$ with $\mathbf{e}_i\in \mathbf{X}_\mathrm{out}$ there exists
a unique vertex $\mathbf{v}_i\in [n]\setminus R$ such that $\mathbf{e}_i\subseteq \big(\{\mathbf{v}_i\}\cup R\big)\times [d]$.
For every $i\leqslant |\mathbf{X}|$ with $\mathbf{e}_i\in \mathbf{X}_\mathrm{out}$ we set
\[ \mathbf{R}_i:=\big\{ r\leqslant |\mathbf{X}|\colon \mathbf{e}_r\in \mathbf{X}_\mathrm{out},
\mathbbm{1}_{\mathcal{F}_r^{\complement}}=1 \text{ and } \mathbf{v}_r=\mathbf{v}_i\big\}; \]
notice that $\mathbf{R}_i$ might be empty, and this happens precisely when
$\mathbbm{1}_{\mathcal{F}_i}=\mathbbm{1}_{\mathcal{E}_i}=1$. Also observe that
if $i,j\leqslant |\mathbf{X}|$ are distinct and satisfy $\mathbf{e}_i,\mathbf{e}_j\in \mathbf{X}_\mathrm{out}$ and
$\mathbbm{1}_{\mathcal{F}_i}=\mathbbm{1}_{\mathcal{F}_j}=1$, then $\mathbf{v}_i\neq \mathbf{v}_j$ and, consequently,
$\mathbf{R}_i\cap \mathbf{R}_j=\emptyset$. In addition, for every $i\leqslant |\mathbf{X}|$ such that
$\mathbf{e}_i\in \mathbf{X}_\mathrm{out}$ and $\mathbbm{1}_{\mathcal{F}_i}-\mathbbm{1}_{\mathcal{E}_i}=1$,
since $\mathcal{E}_i\subseteq \mathcal{F}_i$, we must have that $\mathbbm{1}_{\mathcal{F}_i}=1$ and
$\mathbbm{1}_{\mathcal{E}_i}=0$, and therefore $\mathbf{R}_i\neq\emptyset$.
Summing up, for every $i\leqslant |\mathbf{X}|$ such that $\mathbf{e}_i\in \mathbf{X}_\mathrm{out}$ and
$\mathbbm{1}_{\mathcal{F}_i}-\mathbbm{1}_{\mathcal{E}_i}=1$, we may select $r_i\leqslant |\mathbf{X}|$ with
$\mathbf{e}_{r_i}\in \mathbf{X}_\mathrm{out}$ and $\mathbbm{1}_{\mathcal{F}_{r_i}^{\complement}}=1$ in such a way
that $r_i\neq r_j$ if $i\neq j$. This yields that \eqref{malakia} is satisfied.

Combining the previous observations, we obtain that
\begin{align*}
\sum_{i=1}^{|\mathbf{X}|} \mathbbm{1}_{\mathcal{E}_i} & \ =
\sum_{\substack{i\in \{1,\dots,|\mathbf{X}|\}\\\mathbf{e}_i\in \mathbf{X}_\mathrm{out}}} \mathbbm{1}_{\mathcal{E}_i} \geqslant
\sum_{\substack{i\in \{1,\dots,|\mathbf{X}|\}\\\mathbf{e}_i\in \mathbf{X}_\mathrm{out}}} \mathbbm{1}_{\mathcal{F}_i} -
\sum_{\substack{i\in \{1,\dots,|\mathbf{X}|\}\\\mathbf{e}_i\in \mathbf{X}_\mathrm{out}}} \mathbbm{1}_{\mathcal{F}_i^\complement} =
\sum_{i=1}^{|\mathbf{X}|} \mathbbm{1}_{\mathcal{F}_i} -
\sum_{i=1}^{|\mathbf{X}|} \mathbbm{1}_{\mathcal{F}_i^\complement} + |\mathbf{X}_{\mathrm{in}}| \\
& \stackrel{\eqref{eqkt19}}{=} |\mathbf{X}| - 2 \sum_{i=1}^{|\mathbf{X}|} \mathbbm{1}_{\mathcal{F}_i^\complement} +
m - |\mathbf{X}| = m - 2 \sum_{i=1}^{|\mathbf{X}|} \mathbbm{1}_{\mathcal{F}_i^\complement}.
\end{align*}
Hence,
\begin{align} \label{newnew-sd}
\mathbb{P}\Big[ \big|\Delta_{\mathbf{G}}(R,1)\big|\geqslant \theta|\mathcal{A}| \, \Big|
\, \mathcal{M}_{\pi} \Big] & \geqslant
\mathbb{P}\bigg[ \sum_{i=1}^{|\mathbf{X}|} \mathbbm{1}_{\mathcal{E}_i}>\theta m \, \bigg|\, \mathcal{M}_{\pi} \bigg] \geqslant
\mathbb{P}\bigg[  m - 2 \sum_{i=1}^{|\mathbf{X}|} \mathbbm{1}_{\mathcal{F}_i^\complement}>
\theta m\, \bigg|\, \mathcal{M}_{\pi} \bigg] \\
& = 1 - \mathbb{P}\bigg[  \sum_{i=1}^{|\mathbf{X}|} \mathbbm{1}_{\mathcal{F}_i^\complement}
\geqslant \frac{1-\theta}{2}m \, \bigg|\, \mathcal{M}_{\pi} \bigg]. \nonumber
\end{align}
\begin{claim} \label{claim-sd}
For every integer $t\geqslant 0$, we have
\[  \mathbb{P}\bigg[  \sum_{i=1}^{|\mathbf{X}|} \mathbbm{1}_{\mathcal{F}_i^\complement} \geqslant t
\, \bigg|\, \mathcal{M}_{\pi} \bigg] \leqslant
\binom{m}{t}  \Big(\frac{m}{n-2|R|}\Big)^{t}. \]
\end{claim}
\begin{proof}[Proof of Claim \ref{claim-sd}]
This estimate is essentially known;  see, e.g., \cite[Lemma 2.1]{LS10}. We recall the argument for the
convenience of the reader.

For every $i\in [m]$ let $\mathbf{Y}_i$ denote the boolean random variable defined by setting $\mathbf{Y}_i=1$ if
$i\leqslant |\mathbf{X}|$ and $\mathbbm{1}_{\mathcal{F}_i^\complement}=1$; otherwise, set $\mathbf{Y}_i=0$.
Notice that, everywhere on $\mathcal{M}_\pi$, we have
\[ \sum_{i=1}^{|\mathbf{X}|} \mathbbm{1}_{\mathcal{F}_i^\complement} =\sum_{i=1}^m \mathbf{Y}_i; \]
moreover, for $i\in [m]$ and every $E\in \sigma\big( (\mathbf{Y}_j)_{j<i}\big)$,
\[ \mathbb{P}\big[ [\mathbf{Y}_i=1] \, \big|\, E\cap\mathcal{M}_\pi \big] \leqslant \frac{d|\mathcal{A}|}{dn-|\cup\pi|-2|\mathcal{A}|}
\leqslant \frac{m}{n-2|R|}, \]
where the last inequality follows from the fact that $|\cup\pi|+|\mathcal{A}|=d|R|$. The claim follows from this
observation and a union bound.
\end{proof}
The proof of the lemma is now completed by combining \eqref{newnew-sd} and Claim \ref{claim-sd}.
%Now note that the random variable $\sum_{i=1}^{|\mathbf{X}|} \mathbbm{1}_{\mathcal{F}_i^\complement}$,
%conditioned on $\mathcal{M}_\pi$, is stochastically dominated\footnote{Recall that a nonnegative random variable $X$
%is \emph{stochastically dominated} by a nonnegative random variable $Y$ provided that
%$\mathbb{P}[X\geqslant t]\leqslant \mathbb{P}[Y\geqslant t]$ for all $t\geqslant 0$.} by the binomial
%random variable $\mathrm{Bin}\big(m,\frac{m}{n-2|R|}\big)$; see, e.g., \cite[Lemma 2.1]{LS10}. Consequently,
%\begin{align*}
%\mathbb{P}\bigg[  \sum_{i=1}^{|\mathbf{X}|} \mathbbm{1}_{\mathcal{F}_i^\complement}
%\geqslant \frac{1-\theta}{2}m \, \bigg|\, \mathcal{M}_{\pi} \bigg] & \leqslant
%\mathbb{P}\bigg[ \mathrm{Bin}\Big(m,\frac{m}{n-2|R|}\Big) \geqslant \frac{1-\theta}{2}m \bigg] \\
%&  \leqslant \bigg(\frac{em}{\frac{1-\theta}{2}m}\bigg)^{\frac{1-\theta}{2}m}
%\Big(\frac{m}{n-2|R|}\Big)^{\frac{1-\theta}{2}m}
%\end{align*}
%and the proof is completed.
\end{proof}
Lemma \ref{l1-pap} is the main tool of an inductive argument that will be implemented in Proposition~\ref{p3-pap1} below. To that end, given a nonempty subset $S$ of $[n]$ and a nonnegative integer $\ell$, we need to isolate those (partial) matchings that expose $B_{\mathbf{G}}(S,\ell)$ in a minimal way. This is the content of the following definition.
\begin{definition}
For every nonempty $S\subseteq [n]$ and every nonnegative integer $\ell$, we set
\begin{align*}
\mathcal{M}^S_\ell := \Big\{ \pi\in\mathcal{M}_\mathrm{part} \colon  &
\partial B_{G(\pi)}(S,\ell) \neq\emptyset, \ \partial B_{G(\pi)}(S,\ell+1)=\emptyset, \\
& G(\pi) \text{ has no edge (or self-loop) in } \partial B_{G(\pi)}(S,\ell), \\
& \text{and } \big(B_{G(\pi)}(S,\ell-1)\times [d]\big) \subseteq \cup\pi \Big\}.
\end{align*}
Notice, in particular, that $\mathcal{M}^S_0=\{\emptyset\}$.
\end{definition}
One can roughly think of $\mathcal{M}^S_\ell$ as matchings arising from a breadth-first search started from $S$.
\begin{corollary} \label{c2-pap1}
Let $S\subseteq [n]$ be nonempty, and let $\ell$ be a positive integer such that
$|S|d(d-1)^{\ell-2}\leqslant\frac{n}{6}$ if $\ell\geqslant 2$, and $|S|\leqslant \frac{n}{6}$ if $\ell=1$.
Also let $0<\theta<1$, and let $\pi\in\mathcal{M}^S_{\ell-1}$. Then,
\begin{align*}
\mathbb{P}\Big[ \big|\Delta_{\mathbf{G}}(S,\ell)\big|\geqslant & \
\theta (d-1) \big|\Delta_{G(\pi)}(S,\ell-1)\big| \, \Big| \, \mathcal{M}_{\pi} \Big] \geqslant \\
& 1 - \Big(\frac{4e}{1-\theta}\cdot \frac{d(d-1)^{\ell-1}|S|}{n}
\Big)^{\frac{1-\theta}{2} (d-1) |\Delta_{G(\pi)}(S,\ell-1)|}.
\end{align*}
\end{corollary}
\begin{proof}
If $\ell\geqslant 2$, then the result follows from Lemma \ref{l1-pap} applied for ``$R=V(\pi)=B_{G(\pi)}(S,\ell-1)$",
after observing that in this case $\mathcal{A}=\mathrm{Free}(\pi)$, $n-2|R|\geqslant \frac{n}{2}$,
\[|S|d(d-1)^{\ell-1} \geqslant \big|\mathrm{Free}(\pi)\big|\geqslant (d-1) \big| \Delta_{G(\pi)}(S,\ell-1)\big|\]
and, conditioned on $\mathcal{M}_\pi$,
\[\Delta_{\mathbf{G}}(S,\ell) = \Delta_{\mathbf{G}}\big( B_{\mathbf{G}}(S,\ell-1), 1\big)=
\Delta_{\mathbf{G}}\big( B_{G(\pi)}(S,\ell-1), 1\big)=\Delta_{\mathbf{G}} \big(V(\pi),1\big).  \]
On the other hand, if $\ell=1$, then $\pi=\emptyset$, and the result follows by applying Lemma \ref{l1-pap} to ``$R=S$".
\end{proof}
We define the event
\begin{equation} \label{eq:new_01}
\mathcal{E} := \big[\mathbf{G} \text{ is a simple graph with } \lambda_2(\mathbf{G})\leqslant 2.1\sqrt{d-1}\big];
\end{equation}
notice that, by definition, $\mathcal{E}\subseteq \mathcal{S}$, where $\mathcal{S}$ is as in \eqref{e-ra-s-1}. We also set
\begin{equation}\label{eq:new_02}
L_0 := \big\lfloor(7/15)10^8\big\rfloor+1 \ \ \text{ and } \ \
K := \frac{d-2.1\sqrt{d-1}}{2}\Big( 1.5-1.05\,\frac{\sqrt{d-1}}{d} \Big)^{L_0-1}
\end{equation}
and we observe that
\begin{equation}\label{eq:new_04}
K\geqslant \frac{28}{0.008}=\frac{7}{2}10^3.
\end{equation}
\begin{fact}\label{fact_new_Cheeger_expansion}
Let $S\subseteq [n]$ be nonempty with $d(d-1)^{L_0-2}|S|\leqslant\frac{n}{2}$, where $L_0$ is as in \eqref{eq:new_02}.
Also let $\pi_0\in\mathcal{M}^S_{L_0-1}$ such that $\mathcal{M}_{\pi_0}\cap\mathcal{E}\neq\emptyset$.
Then $|\mathrm{Free}(\pi_0)|\geqslant K |S|$, where $K$ is as in \eqref{eq:new_02}.
\end{fact}
\begin{proof}
Fix $\mu_0\in \mathcal{M}_{\pi_0}\cap\mathcal{E}$. Then, by \eqref{e-cheeger-e1}, we have
\[ \big|B_{G(\pi_0)}(S,L_0-1)\big| = |B_{\mathbf{G}(\mu_0)}(S,L_0-1)\big|
\geqslant\Big( 1.5-1.05\,\frac{\sqrt{d-1}}{d} \Big)^{L_0-1} |S|.\]
The proof is completed using this estimate, the fact that
\[ |\mathrm{Free}(\pi_0)|\geqslant \big|\big\{ \{v,w\}\in E_{G(\mu_0)}\colon v\in S \text{ and } w\in S^{\complement}\big\}\big|,\]
and invoking the choice of $K$.
\end{proof}
\begin{corollary}\label{cor_new_base_step}
Let $S\subseteq [n]$ be nonempty with $d(d-1)^{L_0-2}|S|\leqslant\frac{n}{4}$, where $L_0$ is as in \eqref{eq:new_02}.
Also let $\pi_0\in\mathcal{M}^S_{L_0-1}$ such that $\mathcal{M}_{\pi_0}\cap\mathcal{E}\neq\emptyset$. Then,
\[\mathbb{P}\Big[ \big|\Delta_{\mathbf{G}}(S,L_0)\big|\geqslant \frac{K}{2} |S| \, \Big| \, \mathcal{M}_{\pi_0} \Big] \geqslant
1 - \Big(8e\, \frac{d(d-1)^{L_0-1}|S|}{n} \Big)^{\frac{K}{4} |S|}.\]
\end{corollary}
\begin{proof}
The result follows from Fact \ref{fact_new_Cheeger_expansion} and Lemma \ref{l1-pap} applied for
``$R=V(\pi_0)=B_{G(\pi_0)}(S,L_0-1)$" and ``$\theta = \frac{1}{2}$''; indeed,
in this case we have $\mathcal{A}=\mathrm{Free}(\pi_0)$, $n-2|R|\geqslant \frac{n}{2}$,
\[|S|d(d-1)^{L_0-1} \geqslant \big|\mathrm{Free}(\pi_0)\big|\geqslant K|S|\]
and%, conditioned on $\mathcal{M}_\pi$,
\[\Delta_{\mathbf{G}}(S,L_0) = \Delta_{\mathbf{G}}\big( B_{\mathbf{G}}(S,L_0-1), 1\big)=
\Delta_{\mathbf{G}}\big( B_{G(\pi_0)}(S,L_0-1), 1\big)=\Delta_{\mathbf{G}} \big(V(\pi_0),1\big). \qedhere \]
\end{proof}
\begin{proposition} \label{p3-pap1}
Set
\begin{equation} \label{e-pap1-p1}
\eta=\eta(d):= \frac{1}{12^{2} e^{3} (d-1)^{2L_0+2}},
\end{equation}
where $L_0$ is as in \eqref{eq:new_02}.
Let $S\subseteq [n]$ be nonempty with $|S|\leqslant \eta n$,
let $\pi_0\in\mathcal{M}^S_{L_0-1}$ such that $\mathcal{M}_{\pi_0}\cap\mathcal{E}\neq\emptyset$, and set
\begin{equation} \label{e-pap-p2}
T:= \Big\lfloor \log_{d-1}\Big(\frac{\eta n}{|S|}\Big)\Big\rfloor.
\end{equation}
Then,
\begin{equation} \label{eqkt22}
\mathbb{P} \Big[ \big|\partial B_{\mathbf{G}}(S,L_0+\ell)\big| \geqslant
28  (d-1)^\ell |S| \text{ for all } \ell\in [T]  \, \Big| \, \mathcal{M}_{\pi_0} \Big]
 \geqslant 1- \Big(\frac{|S|}{en}\Big)^{3|S|}.
\end{equation}
\end{proposition}
\begin{proof}
Set $r:=\frac{1}{\sqrt{2}}$ and
\begin{equation} \label{e-pap-p3}
\begin{cases}
\delta_\ell := r^{\ell} & \text{if } \, \ell\in \big\{1,\dots,\lfloor 2T/3 \rfloor\big\}, \\
\delta_\ell :=  r^{1+2T-2\ell} & \text{if } \, \ell\in \big\{ \lfloor 2T/3 \rfloor+1,\dots,T\big\},
\end{cases}
\ \ \ \
\begin{cases}
\alpha_0:=\frac{1}{2}, \\
\alpha_\ell := \prod_{t=0}^{\ell}(1-\delta_t) & \text{if } \, \ell\in \{1,\dots,T\}.
\end{cases}
\end{equation}
Observe that $(\alpha_\ell)_{\ell=0}^T$ is decreasing, and
\begin{equation} \label{"douleia"}
\alpha_T\geqslant \frac{1}{2}\Big(\prod_{t=1}^{\infty}(1-r^t)\Big)
\Big(\prod_{t=0}^{\infty}(1-r r^{2t})\Big) \geqslant 0.008.
\end{equation}
Since $\Delta_{\mathbf{G}}(S,\ell)$ is a subset of $\partial B_{\mathbf{G}}(S,\ell)$ everywhere
in our probability space, by \eqref{eq:new_04}, the desired estimate \eqref{eqkt22} will follow once we show that
\begin{equation} \label{eqkt23}
\mathbb{P}\Big[ \big|\Delta_{\mathbf{G}}(S,L_0+\ell)\big| \geqslant
0.008 K (d-1)^\ell |S| \text{ for all } \ell\in [T]\Big]\geqslant 1- \Big(\frac{|S|}{en}\Big)^{3|S|},
\end{equation}
where $K$ is as in \eqref{eq:new_02}. For every nonnegative integer $\ell$, set
\begin{align*}
\widetilde{\mathcal{M}}^S_\ell & := \big\{ \pi\in\mathcal{M}^S_{L_0+\ell}\colon
\mathcal{M}_{\pi_0}\subseteq\mathcal{M}_\pi\text{ and }
\big| \Delta_{G(\pi)}(S,L_0+\ell)\big| \geqslant \alpha_\ell K (d-1)^\ell|S|\big\}, \\
\widehat{\mathcal{M}}^S_\ell & := \big\{\pi\in\mathcal{M}^S_{L_0+\ell} \colon
\mathcal{M}_{\pi_0}\subseteq\mathcal{M}_\pi\text{ and }
\big|\Delta_{G(\pi)}(S,L_0+t)\big| \geqslant \alpha_t K (d-1)^t|S| \text{ for all } t\in\{0,\dots,\ell\}\big\}.
\end{align*}
%{\color{red}Condition $\mathcal{M}_{\pi_0}\subseteq\mathcal{M}_\pi$ can be substituted equivalently by
%$\pi_0\subseteq\pi$ or $\mathcal{M}_{\pi_0}\cap\mathcal{M}_\pi\neq\emptyset$. Although $\pi_0\subseteq\pi$ is
%kind of more basic, the one I put there describes more directly what is going on with the events of the probability space. Just mentioning %it Prof. Pandelis to pick whatever you prefer.}
Clearly, $\widehat{\mathcal{M}}^S_\ell\subseteq\widetilde{\mathcal{M}}^S_\ell$
for all $\ell\in \{0,\ldots,T\}$. Moreover, by Corollary \ref{cor_new_base_step}, the choice of $\alpha_0$ and the fact that $\frac{d}{d-1}\leqslant \frac{3}{2}$, we have
\[\mathbb{P}\Big[ \big|\Delta_{\mathbf{G}}(S,L_0)\big|\geqslant \alpha_0 K |S| \, \Big| \, \mathcal{M}_{\pi_0} \Big] \geqslant
  1 - \Big(12e\, \frac{(d-1)^{L_0}|S|}{n} \Big)^{\frac{K}{4} |S|}.\]
On the other hand, by Corollary \ref{c2-pap1} and the fact that $\frac{d}{d-1}\leqslant \frac{3}{2}$,
for every positive integer $\ell$ and every $\pi\in \widetilde{\mathcal{M}}^S_{\ell-1}$, we have
\[ \mathbb{P}\Big[ \big|\Delta_{\mathbf{G}}(S,L_0+\ell)\big| \geqslant  \alpha_\ell K (d-1)^\ell |S| \, \Big| \,
\mathcal{M}_{\pi} \Big]
\geqslant 1 - \Big(\frac{6e}{\delta_\ell} \cdot
\frac{(d-1)^{L_0+\ell}|S|}{n}\Big)^{\frac{\delta_\ell}{2} \alpha_{\ell-1}K (d-1)^{\ell} |S|}.\]
Next, observe that for every $\ell\in\{0,\dots,T\}$, any two distinct $\pi_1,\pi_2\in \mathcal{M}^S_{L_0+\ell}$ are incomparable
under inclusion which, in turn, implies that $\mathcal{M}_{\pi_1}\cap\mathcal{M}_{\pi_2}=\emptyset$;
consequently, we obtain that
\[\mathbb{P}\bigg[ \bigcup_{\pi\in \widehat{\mathcal{M}}^S_0} \mathcal{M}_{\pi} \, \bigg| \, \mathcal{M}_{\pi_0} \bigg] \geqslant
  1 - \Big(12e\, \frac{(d-1)^{L_0}|S|}{n}
\Big)^{\frac{K}{4} |S|}\]
and, for every $\ell\in[T]$,
\[\mathbb{P}\bigg[ \bigcup_{\pi\in \widehat{\mathcal{M}}^S_\ell} \mathcal{M}_{\pi} \, \bigg| \,
\bigcup_{\pi\in \widehat{\mathcal{M}}^S_{\ell-1}} \mathcal{M}_{\pi} \bigg]
\geqslant 1 - \Big(\frac{6e}{\delta_\ell} \cdot
\frac{(d-1)^{L_0+\ell}|S|}{n}\Big)^{\frac{\delta_\ell}{2} \alpha_{\ell-1}K (d-1)^{\ell} |S|}.\]
Hence, setting $\widetilde{\mathcal{M}}^S_{-1} = \widehat{\mathcal{M}}^S_{-1} := \{\pi_0\}$,
by a union bound  we have
\begin{align}
\mathbb{P}\bigg[  \bigcap_{\ell=0}^T  \Big[ \big|\Delta_{\mathbf{G}}(S, & L_0+\ell)\big| \geqslant
0.008 K (d-1)^\ell |S|  \big]  \, \bigg| \, \mathcal{M}_{\pi_0}  \bigg] \label{e-pap-p4} \\
&\geqslant
\mathbb{P}\bigg[ \bigcap_{\ell=0}^T \Big[ \big|\Delta_{\mathbf{G}}(S,L_0+\ell)\big| \geqslant
\alpha_\ell K(d-1)^\ell |S| \Big]  \, \bigg| \, \mathcal{M}_{\pi_0} \bigg] \nonumber\\
& = \mathbb{P}\bigg[ \bigcap_{\ell=1}^T \bigcup_{\pi\in \widehat{\mathcal{M}}^S_\ell} \mathcal{M}_{\pi} \, \bigg| \, \mathcal{M}_{\pi_0} \bigg] =
\prod_{\ell=1}^{T} \mathbb{P}\bigg[ \bigcup_{\pi\in\widehat{\mathcal{M}}^S_\ell} \mathcal{M}_{\pi} \, \bigg| \,
\bigcup_{\pi\in \widehat{\mathcal{M}}^S_{\ell-1}} \mathcal{M}_{\pi} \bigg] \nonumber \\
& \geqslant 1 -\Big(12e\, \frac{(d-1)^{L_0}|S|}{n} \Big)^{\frac{K}{4} |S|}
- \sum_{\ell=1}^{T}  \Big(\frac{6e}{\delta_\ell} \cdot
\frac{(d-1)^{L_0+\ell}|S|}{n}\Big)^{\frac{\delta_\ell}{2} \alpha_{\ell-1}K (d-1)^{\ell} |S|}. \nonumber
\end{align}
Thus, in order to prove \eqref{eqkt23}, it suffices to show that the quantities subtracted  in the right-hand side of \eqref{e-pap-p4}
add up to at most $\binom{|S|}{e n}^{3|S|}$. This, in turn, would follow once we show that
\begin{equation}\label{eq:new_03}
\bigg(12 e^2 (d-1)^{L_0} \Big(\frac{|S|}{en}\Big)^{1-\frac{12}{K}} \bigg)^{\frac{K}{4}|S|} \leqslant \frac{1}{2}
\end{equation}
and, for every $\ell\in[T]$,
\begin{equation} \label{e-pap-p5}
\bigg(6e^2(d-1)^{L_0}\frac{(d-1)^{\ell}}{\delta_\ell} \cdot
\Big(\frac{|S|}{en}\Big)^{1-\frac{6}{\delta_\ell\alpha_{\ell-1}K(d-1)^{\ell}}}
\bigg)^{\frac{\delta_\ell}{2} \alpha_{\ell-1}K (d-1)^{\ell} |S|}
\leqslant \frac{1}{2^{\ell+1}}.
\end{equation}

First we argue for \eqref{eq:new_03}. Since $K\geqslant 6\geqslant 4$ and $|S|\leqslant\eta n$, it suffices to show that
\[ 12 e^2 (d-1)^{L_0} (\eta/e)^{\frac{1}{2}}\leqslant\frac{1}{2}, \]
or, equivalently, that $\eta\leqslant 24^{-2}e^{-3}(d-1)^{-2L_0}$.
This follows from the choice of $\eta$ in \eqref{e-pap1-p1} and the fact that $d\geqslant 3$.

We proceed to the proof of \eqref{e-pap-p5}. By \eqref{"douleia"}, it is enough to show that for every $\ell\in[T]$,
\begin{equation} \label{eq:new_05}
0.004 K \delta_\ell (d-1)^\ell |S|\geqslant \ell+1
\end{equation}
and
\begin{equation} \label{eq:new_06}
6e^2(d-1)^{L_0}\frac{(d-1)^\ell}{\delta_\ell} \Big(\frac{|S|}{en}\Big)^{1-\frac{6}{0.008 K \delta_\ell (d-1)^\ell}}
\leqslant\frac{1}{2}.
\end{equation}
To this end, set $m_\ell:= \log_{r}(\delta_\ell)$ for every $\ell\in[T]$.
\begin{claim}\label{claim_new_1}
 For every $\ell\in[T]$, we have $\ell - \frac{m_\ell}{2}\geqslant \frac{\ell}{2}-\frac{1}{2}$.
\end{claim}
\begin{proof}[Proof of Claim \ref{claim_new_1}]
Let $\ell\in[T]$. If $\ell\leqslant\lfloor 2T/3\rfloor$, then $m_\ell = \ell$. On the other hand,
if $\ell \geqslant \lfloor 2T/3\rfloor+1$, then $T\leqslant 3\ell/2$ and $m_\ell = 2T -2\ell +1$;
therefore, $\ell - \frac{m_\ell}{2} = 2\ell-T-\frac{1}{2}\geqslant \frac{\ell}{2}-\frac{1}{2}$, as desired.
\end{proof}
By Claim \ref{claim_new_1} and  \eqref{eq:new_04}, we have for every $\ell\in [T]$,
\begin{equation} \label{"douleia2"}
2^{\ell - \frac{m_\ell}{2}}\geqslant \frac{6}{0.008 K}(\ell+1)
\geqslant \max\Big\{ \frac{2}{0.008 K}(\ell+1), \frac{12}{0.008 K}\Big\}.
\end{equation}
\begin{claim}\label{claim_new_3}
For every $\ell\in[T]$, we have $\ell + \frac{m_\ell}{2} - T +\frac{6T}{0.008 K 2^{\ell-\frac{m_\ell}{2}}}\leqslant1$.
\end{claim}
\begin{proof}[Proof of Claim \ref{claim_new_3}]
Fix $\ell\in[T]$. By Claim \ref{claim_new_1}, it is enough to show that $\ell + \frac{m_\ell}{2} - T +\frac{6\sqrt{2}T}{0.008 K 2^{\frac{\ell}{2}}}\leqslant1$.

First, assume that $\ell \leqslant \lfloor T/2 \rfloor$. By \eqref{eq:new_04}, we see that $\frac{0.008K}{6}\geqslant 4$. Hence, using the fact that in this case $m_\ell=\ell$, we obtain that
\[ \ell + \frac{m_\ell}{2} - T +\frac{6\sqrt{2}T}{0.008 K 2^{\ell/2}}\leqslant-\frac{T}{4} + \frac{T}{0.008 K/6 }\leqslant0.\]

Next, assume that $\ell\in\{\lfloor T/2\rfloor+1,\dots,\lfloor2T/3\rfloor\}$. Then, $m_\ell = \ell$.
Since the function $f(x)=x2^{-\frac{x}{4}}$ is bounded above by $2\sqrt{2}$, by \eqref{eq:new_04}, we have
\[ 1\geqslant \frac{6\sqrt{2}T}{0.008 K 2^{T/4}} \geqslant \frac{6\sqrt{2}T}{0.008 K 2^{\ell/2}}
\geqslant \ell +\frac{m_\ell}{2} -T + \frac{6\sqrt{2}T}{0.008 K 2^{\ell/2}}. \]

Finally, assume that $\ell \in \{\lfloor2T/3\rfloor+1,\dots,T\}$.
Then, $m_\ell = 2T - 2\ell +1$. Using, in this case, the fact that the function $f(x)=x2^{-\frac{x}{3}}$ is bounded above by $\frac{7}{6}\sqrt{2}$, by \eqref{eq:new_04}, we conclude that
\[ 1\geqslant \frac{1}{2} + \frac{6\sqrt{2}T}{0.008 K 2^{T/3}} \geqslant
\frac{1}{2} + \frac{6\sqrt{2}T}{0.008 K 2^{\ell/2}}=
\ell +\frac{m_\ell}{2} -T + \frac{6\sqrt{2}T}{0.008 K 2^{\ell/2}}. \qedhere \]
\end{proof}
We are now ready to finish the proof.
First note that \eqref{eq:new_05} follows from \eqref{"douleia2"}
after observing that $\delta_\ell (d-1)^\ell \geqslant 2^{\ell-\frac{m_\ell}{2}}$.
In order to show \eqref{eq:new_06}, we fix $\ell\in [T]$. Since $d\geqslant3$, $(d-1)^T|S|\leqslant \eta n$
and $\delta_\ell^{-1} = 2^{\frac{m_\ell}{2}} \leqslant (d-1)^{\frac{m_\ell}{2}}$, by \eqref{"douleia2"} and
Claim \ref{claim_new_3}, we have
\begin{align*}
 6e^2&(d-1)^{L_0}\frac{(d-1)^\ell}{\delta_\ell} \Big(\frac{|S|}{en}\Big)^{1-\frac{6}{0.008 K \delta_\ell (d-1)^\ell}} \\
  & \leqslant 6e^2(d-1)^{L_0}(d-1)^{\ell+\frac{m_\ell}{2}-T+\frac{6T}{0.008 K \delta_\ell (d-1)^\ell}}
  \Big(\frac{\eta}{e}\Big)^{1-\frac{6}{0.008 K \delta_\ell (d-1)^{\ell}}} \\
  & \leqslant 6e^2(d-1)^{L_0}(d-1)^{\ell+\frac{m_\ell}{2}-T+\frac{6T}{0.008 K 2^{\ell-\frac{m_\ell}{2}}}}
  \Big(\frac{\eta}{e}\Big)^{1-\frac{6}{0.008 K (d-1)^{\ell-\frac{m_\ell}{2}}}} \\
  & \leqslant 6e^2(d-1)^{L_0+1} \Big(\frac{\eta}{e}\Big)^{1/2} \leqslant\frac{1}{2},
\end{align*}
where the last inequality holds by the choice of $\eta$ in \eqref{e-pap1-p1}. The proof of Proposition \ref{p3-pap1}
is thus completed.
\end{proof}

\begin{corollary}\label{cor_new_cond_E}
Set
\begin{align} \label{eq:new_08}
\mathcal{G}_1:= \Big\{ G\in G(n,d) \colon & \lambda_2(G) \leqslant 2.1\sqrt{d-1},
\text{ and for every nonempty } S\subseteq [n] \\
& \text{and every integer } \ell\geqslant 0, \text{ if } |S|(d-1)^\ell\leqslant \eta(d-1)^{L_0} n, \nonumber \\
& \text{then } \big| B_G(S,\ell)\big| \geqslant \frac{1}{(d-1)^{L_0}} |S|(d-1)^\ell \Big\}, \nonumber
\end{align}
where $\eta$ is as in \eqref{e-pap1-p1} and $L_0$ is as in \eqref{eq:new_02}. Then, denoting by
$\mathbb{P}$ the uniform probability measure on $G(n,d)$, we have
\begin{equation} \label{eq:new_07}
\mathbb{P}[\mathcal{G}_1\big]  \geqslant 1- O_d\Big(\frac{1}{n^\tau}\Big),
\end{equation}
where $\tau=\tau(d)>0$.
\end{corollary}
\begin{proof}
For every nonempty $S\subseteq[n]$, define the event
\begin{align*}
\mathcal{A}_S:= \Big[ \big| B_\mathbf{G}(S,L_0+\ell)\big| \geqslant 28 |S|(d-1)^\ell
\text{ for every integer } \ell\geqslant 0 \text{ with } |S|(d-1)^\ell\leqslant \eta n\Big].
\end{align*}
By Proposition \ref{p3-pap1}, for every $\pi\in\mathcal{M}^S_{L_0-1}$ with $\mathcal{M}_\pi\cap\mathcal{E}\neq\emptyset$
we have $\mathbb{P}\big[\mathcal{A}_S^\complement \, \big| \, \mathcal{M}_\pi \big] \leqslant \big(\frac{|S|}{en}\big)^{3|S|}$.
Therefore, for every $\pi\in\mathcal{M}^S_{L_0-1}$,
\[ \mathbb{P} \big[\mathcal{A}_S^\complement \cap \mathcal{E} \, \big| \, \mathcal{M}_\pi \big]
\leqslant \Big(\frac{|S|}{en}\Big)^{3|S|}  \]
and consequently, by \eqref{e-ra-s-1} and Theorem \ref{Friedman},
\[ \mathbb{P} \big[\mathcal{A}_S^\complement \, \big| \, \mathcal{E} \big] =
\mathbb{P}[\mathcal{E}]^{-1}\sum_{\pi\in\mathcal{M}^S_{L_0-1}}
\mathbb{P}\big[\mathcal{A}_S^\complement \cap \mathcal{E} \, \big| \, \mathcal{M}_\pi \big]\,
\mathbb{P}[\mathcal{M}_\pi] \leqslant
\mathbb{P}[\mathcal{E}]^{-1} \Big(\frac{|S|}{en}\Big)^{3|S|} =O_d\bigg( \Big(\frac{|S|}{en}\Big)^{3|S|}\bigg). \]
By a union bound,  we conclude that
\begin{equation} \label{eq:new_09}
\mathbb{P} \big[ \mathbf{G} \in\mathcal{G}_1  \, \big| \, \mathcal{E} \big] \geqslant
\mathbb{P} \bigg[ \bigcap_{\emptyset\neq S\subseteq[n]} \mathcal{A}_S \, \bigg| \, \mathcal{E} \bigg]
 \geqslant 1- O_d\Big(\frac{1}{n^2}\Big).
\end{equation}
The result follows by invoking again \eqref{e-ra-s-1} and Theorem \ref{Friedman}, and using the fact that the distribution of $\mathbf{G}$
conditioned on $\mathcal{S}$ coincides with that of a random $d$-regular graph uniformly sampled from~$G(n,d)$.
\end{proof}
The following lemma is the last piece of information that is needed in order to complete the proof of part (a)
of Proposition \ref{lemma:long-range-expansion}.
\begin{lemma} \label{lemCheeger}
Let $0<\delta< \frac{3}{4}$, and set
\begin{equation} \label{equat-cheeger-e1}
\ell_*=\ell_*(\delta):= \Big\lceil \log_{1.0016}\Big(\frac{3}{4\delta}\Big)\Big\rceil \ \ \ \text{ and } \ \ \
\gamma=\gamma(\delta):= \Big( \frac{1.0016}{d-1}\Big)^{\ell_*}.
\end{equation}
Let $G$ be a (deterministic) $d$-regular graph on $[n]$ with $h(G)\geqslant 0.0048\, d$. Then for every $A\subseteq [n]$
with $|A|\geqslant\delta n$ and every positive integer $\ell$, we have
\[ \big| B_G(A,\ell)\big| \geqslant \min\Big\{\frac{3n}{4},\gamma (d-1)^\ell|A|\Big\}. \]
\end{lemma}
\begin{proof}
We first observe that for every nonempty $B\subseteq [n]$ with $|B|\leqslant \frac{3n}{4}$, we have
\begin{equation} \label{eqn02}
\big| B_G(B,1)\big|\geqslant 1.0016\, |B|.
\end{equation}
Indeed, if $|B|\leqslant \frac{n}{2}$, then \eqref{eqn02} follows from the $d$-regularity of $G$ and
the fact that $h(G)\geqslant 0.0048\, d$. So, assume that $\frac{n}{2}< |B|\leqslant \frac{3n}{4}$;
notice that $0<|B^\complement|\leqslant\frac{n}{2}$ and $|B^\complement|\geqslant\frac{1}{3}|B|$.
Using again the fact that $h(G)\geqslant 0.0048\,d$, we see that there exist at least
$0.0048\, d|B^\complement|$ edges connecting nodes in $B$ to nodes in $B^\complement$.
By the $d$-regularity of $G$, we conclude that $|\partial B_G(B,1)|\geqslant 0.0048\, |B^\complement| \geqslant 0.0016\, |B|$
that clearly implies \eqref{eqn02}.

Applying \eqref{eqn02} successively we see, in particular, that for every nonempty $B\subseteq [n]$ and
every nonnegative integer $\ell$, we have
\begin{equation}\label{eqn03}
\big| B_G(B,\ell)\big| \geqslant \min\Big\{\frac{3n}{4}, 1.0016^\ell|B|\Big\}.
\end{equation}

Now let $A$ and $\ell$ be as in the statement of the lemma. If $\ell\leqslant\ell_*$, then
\[1.0016^\ell|A|= \Big(\frac{1.0016}{d-1}\Big)^\ell (d-1)^\ell|A| \stackrel{\eqref{equat-cheeger-e1}}{\geqslant}
\gamma (d-1)^\ell|A|, \]
and so, the result follows from \eqref{eqn03}. Otherwise, if $\ell>\ell_*$, then
$1.0016^\ell|A| \geqslant 1.0016^{\ell_*}\delta n\geqslant\frac{3}{4}n$ and
\[\gamma (d-1)^\ell|A|\geqslant \gamma (d-1)^{\ell_*}\delta n = 1.0016^{\ell_*}
\delta n\geqslant\frac{3}{4}n, \]
and the result also follows from \eqref{eqn03}.
\end{proof}
We are finally in a position to complete the proof. Specifically, by Corollary \ref{cor_new_cond_E}, it is enough to show that if $G\in\mathcal{G}_1$---where $\mathcal{G}_1$ is as in \eqref{eq:new_08}---then
for every nonempty $S\subseteq [n]$ and every positive integer $\ell$, we have
\[ \big|B_{G}(S,\ell)\big|\geqslant \min\Big\{ \frac{3n}{4}, \alpha (d-1)^\ell |S| \Big\}, \]
where $\alpha$ is as in \eqref{era-1}.

So, fix $G\in\mathcal{G}_1$. By \eqref{e-cheeger-e1}, we see that
\begin{equation}\label{eq:new_10}
h(G)\geqslant \frac{d-2.1\sqrt{d-1}}{2}\geqslant 0.0048 d.
\end{equation}
Let $\eta$ be as in \eqref{e-pap1-p1}, let $L_0$ be as in \eqref{eq:new_02}, and set
\[ \delta := \frac{\eta}{d-1}=\frac{1}{12^{2} e^{3} (d-1)^{2L_0+3}}, \ \ \
\ell_* := \Big\lceil \log_{1.0016}\Big(\frac{3}{4\delta}\Big)\Big\rceil \ \ \ \text{ and } \ \ \
\gamma :=\Big(\frac{1.0016}{d-1}\Big)^{\ell_*}. \]
Notice that $\alpha \leqslant \frac{\gamma}{(d-1)^{L_0}}$. Let $S\subseteq [n]$ be nonempty, and let $\ell$ be a positive integer.
If $|S|\geqslant\delta n$, then, by Lemma \ref{lemCheeger} and \eqref{eq:new_10}, we have
\[ \big|B_{G}(S,\ell)\big| \geqslant
\min\Big\{\frac{3n}{4},\gamma (d-1)^\ell|S|\Big\} \geqslant \min\Big\{\frac{3n}{4},\alpha (d-1)^\ell|S|\Big\}.\]
So, assume that $|S|<\delta n$. Set $\ell_1:= \min\big\{k\in\mathbb{N}\colon (d-1)^k|S|> \eta(d-1)^{L_0} n\big\}-1$,
and observe that
\begin{equation} \label{eqn04}
 \eta(d-1)^{L_0-1}n \leqslant (d-1)^{\ell_1}|S| \leqslant \eta(d-1)^{L_0} n.
\end{equation}
Thus, if $\ell\leqslant \ell_1$, then, since $G\in\mathcal{G}_1$, we have
\[ \big| B_{G}(S,\ell)\big| \geqslant \frac{1}{(d-1)^{L_0}} (d-1)^\ell |S|
\geqslant \alpha (d-1)^\ell |S|
= \min\Big\{\frac{3n}{4},\alpha (d-1)^\ell |S|\Big\},\]
%{\color{red}where the last equality follows from the fact that
%\[ \alpha (d-1)^\ell |S|\leqslant (d-1)^{\ell_1}|S|\stackrel{\eqref{eqn04}}{\leqslant}\eta(d-1)^{L_0} n\leqslant\frac{3n}{4}. \]
%why we need this? we can just pass by lower or equal to\\}
Finally, assume that $\ell>\ell_1$, and set $A:=B_{G}(S,\ell_1)$. By \eqref{eqn04} and the fact that
$G\in\mathcal{G}_1$, we have
\begin{equation}\label{eqn05}
  |A|\geqslant \frac{1}{(d-1)^{L_0}} (d-1)^{\ell_1}|S|\stackrel{\eqref{eqn04}}{\geqslant}\delta n.
\end{equation}
Applying Lemma \ref{lemCheeger}, we conclude that
\[ \big|B_{G}(S,\ell)\big| = \big|B_{G}(A,\ell-\ell_1)\big|
\geqslant \min\Big\{\frac{3n}{4},\gamma (d-1)^{\ell-\ell_1}|A|\Big\}
\stackrel{\eqref{eqn05}}{\geqslant}
\min\Big\{\frac{3n}{4},\alpha (d-1)^\ell|S|\Big\}, \]
as desired.

\subsection{Proof of part (b)}

As in the proof of part (a), for the rest of this proof we fix a pair $n>d\geqslant 6$ of integers such that $dn$ is even. We also fix a $d$-regular graph $G$ on $[n]$ such that $\lambda(G)\leqslant 2.1\sqrt{d-1}$; for notational simplicity, we shall drop the dependence on $G$ and we shall denote $\lambda(G)$ by $\lambda$, and the adjacency matrix of $G$ by $A$.

Recall that our goal is to show that the graph $G$ satisfies part (\hyperref[Part-B]{$B$}) of property
$\mathrm{Expan}(\alpha,\varepsilon,L)$, where $\alpha,\varepsilon,L$ are as in $\eqref{era-1}$,
$\eqref{era-2}$ and $\eqref{era-3}$, respectively. To this end, we fix a nonempty subset $S$ of $[n]$ and a
positive integer $\ell$ such that $\alpha (d-1)^{\ell-1}|S|\leqslant \frac{3n}{4}$, and we set
\[ T:=\Big\{e\in E_G\colon \big|\{v\in S\colon \dist_G(v,e)\leqslant \ell-1\}\big|\geqslant L(d-1-\varepsilon)^\ell\Big\}. \]
We need to find a vertex $v\in S$ such that
$\big|\{e\in T\colon \dist_G(v,e)\leqslant \ell-1\}\big|\leqslant L(d-1-\varepsilon)^\ell$.

Assume towards a contradiction that this is not possible. That is,
\begin{equation} \label{con-part-b}
\big|\{e\in T\colon \dist_G(v,e)\leqslant \ell-1\}\big|> L(d-1-\varepsilon)^\ell \ \ \ \text{ for all } v\in S.
\end{equation}
By the choice of $L$ in \eqref{era-3}, this implies that $\ell\geqslant 2$.
We will need the following upper bound~on~$|T|$.
\begin{fact} \label{f1-part-b}
We have that
\[ |T| \leqslant |S|\, \frac{d}{L(d-2)}\, \Big( \frac{d-1}{d-1-\varepsilon}\Big)^\ell. \]
\end{fact}
\begin{proof}
We proceed by double counting. On the one hand, by the definition of $T$, we have
\[ \sum_{v \in S,\, e \in T} \mathbbm{1}_{[\mathrm{dist}_G(v,e) \leqslant \ell-1]} =
\sum_{e \in T} \big|\{v \in S\colon \mathrm{dist}_G(v, e) \leqslant \ell-1 \}\big|
\geqslant |T|L(d-1-\eps)^\ell.  \]
On the other hand, by the $d$-regularity of $G$,
\[ \sum_{v \in S,\, e \in T}\mathbbm{1}_{[\mathrm{dist}_G(v,e) \leqslant \ell-1]} =
\sum_{v \in S} \big|\{e \in T\colon \mathrm{dist}_G(v, e) \leqslant \ell-1 \}\big|
\leqslant  |S| \frac{d}{d-2} (d-1)^\ell. \]
The result follows by combining the above two inequalities.
\end{proof}
We will also need the following fact.
\begin{fact} \label{f2-part-b}
Let $y\in\mathbb{R}^n$ with $\|y\|_2=1$ and such that $y_1+\dots+y_n=0$. Then,
\[ \bigg\|\sum_{k = 1}^{\ell} A^k y \bigg\|_2^2 \leqslant
4 \big( 4.41(d-1)\big)^\ell. \]
\end{fact}
\begin{proof}
Our assumptions imply that $\big\|A^ky\big\|_2 \leqslant \lambda^k$ for all $k\in [\ell]$. Hence,
if $\lambda\geqslant \sqrt{d-1}$, then
\[ \bigg\|\sum_{k = 1}^{\ell} A^k y \bigg\|_2^2 \leqslant
\bigg(\sum_{k = 1}^{\ell} \big\| A^k y \big\|_2\bigg)^2 \leqslant
\Big(\frac{\lambda}{\lambda-1}\Big)^2\, \lambda^{2\ell}, \]
and the desired estimate follows from this inequality and the fact that, in this case, we have that
$\sqrt{5}\leqslant \lambda \leqslant 2.1\sqrt{d-1}$. On the other hand, if $\lambda\leqslant \sqrt{d-1}$, then
\[ \bigg\|\sum_{k = 1}^{\ell} A^k y \bigg\|_2^2 \leqslant
\bigg(\sum_{k = 1}^{\ell} \big\| A^k y \big\|_2\bigg)^2 \leqslant
\ell^2 (d-1)^{\ell} \leqslant 4 \big( 4.41(d-1)\big)^\ell. \qedhere\]
\end{proof}
Let $R:=\big\{ v\in [n]\colon v\in e\in T\big\}$ denote the set of vertices
spanned by the edges in $T$. Since $G$ is $d$-regular, we see that
\[ \frac{|T|}{d} \leqslant |R| \leqslant 2|T|. \]
\begin{claim} \label{cl3-part-b}
We have
\[ \frac{|R|}{|R^\complement|}\leqslant 1 \ \ \ \text{ and } \ \ \
\frac{|R^\complement|}{|R|} \geqslant n\, \frac{(d-2)L}{4d\, |S|} \Big(\frac{d-1-\varepsilon}{d-1}\Big)^\ell. \]
\end{claim}
\begin{proof}
By Fact \ref{f1-part-b}, we have
\[ |R|\leqslant 2|T| \leqslant \frac{2d}{(d-2)L}\, |S|\, \Big(\frac{d-1}{d-1-\varepsilon}\Big)^\ell.\]
On the other hand, since $\alpha|S|(d-1)^{\ell-1}\leqslant \frac{3n}{4}$ and $\ell\geqslant 2$, we have
\[ |R|\leqslant  \frac{3n}{4} \cdot \frac{2d}{\alpha (d-2) L} \cdot \frac{d-1}{(d-1-\varepsilon)^2} \leqslant \frac{n}{2}, \]
where the last inequality  follows from the choice of $\alpha,\varepsilon$ and $L$.
In particular, $|R^\complement| \geqslant \frac{n}{2}$. The claim follows by combining these estimates.
\end{proof}
Define the vector $x\in \mathbb{R}^n$ by the rule
\[ x:= \sqrt{\frac{|R^{\complement}|}{|R|}}\, \frac{1}{\sqrt{n}}\, \mathbbm{1}_{R} -
\sqrt{\frac{|R|}{|R^{\complement}|}}\, \frac{1}{\sqrt{n}}\, \mathbbm{1}_{R^{\complement}} \]
and notice that
\begin{align} \label{eq::01}
\bigg\| \sum_{k=1}^{\ell} A^k x\bigg\|_2^2 &
\geqslant \sum_{v\in S} \bigg(  \sum_{k=1}^{\ell} (A^k x)_v\bigg)^2
= \sum_{v\in S} \bigg( \sum_{w\in [n]} \sum_{k=1}^{\ell} (A^k)_{v,w}\, x_w \bigg)^2  \\
&  = \sum_{v\in S} \bigg( \sum_{w\in R} \sqrt{\frac{|R^\complement|}{|R|}}\, \frac{1}{\sqrt{n}} \,
\sum_{k=1}^{\ell} (A^k)_{v,w} -
\sum_{w\in R^\complement} \sqrt{\frac{|R|}{|R^\complement|}}\, \frac{1}{\sqrt{n}}
\sum_{k=1}^{\ell} (A^k)_{v,w}\bigg)^2.  \nonumber
\end{align}
Also observe that for every $v,w\in [n]$ the quantity $\sum_{k=1}^{\ell} (A^k)_{v,w}$ counts the number of paths
of length at most $\ell$ from $v$ to $w$; therefore, for every $v\in [n]$,
\begin{equation} \label{eq::03}
\sum_{w\in R^\complement}\sum_{k=1}^{\ell} (A^k)_{v,w} \leqslant
\sum_{w\in [n]}\sum_{k=1}^{\ell} (A^k)_{v,w} \leqslant \frac{d}{d-1}\, d^\ell.
\end{equation}
On the other hand, for every $v,w\in [n]$ we have that $\sum_{k=1}^{\ell} (A^k)_{v,w}\geqslant
\mathbbm{1}_{B_G(v,\ell)}(w)$. (Here, as in the previous subsection, $B_G(v,\ell)$ denotes
the set $\big\{v'\in [n]\colon \mathrm{dist}_G(v,v')\leqslant\ell\big\}$.) Thus, for every $v\in S$, since the set
of edges $\big\{e\in T\colon \mathrm{dist}_G(e,v)\leqslant\ell-1\big\}$ is contained in the set of edges of the induced
subgraph of $G$ on $B_G(v,\ell)$, by the $d$-regularity of $G$ and \eqref{con-part-b}, we obtain that
\begin{equation}\label{eq::02}
 \sum_{w\in R}\sum_{k=1}^{\ell} (A^k)_{v,w}\geqslant \frac{L}{d}(d-1-\varepsilon)^\ell.
\end{equation}
Finally, by the choice of $L$ in $\eqref{era-3}$, we have that
\begin{equation} \label{eq::04}
\frac{L^{\frac{3}{2}}\, \sqrt{d-2}\, (d-1-\varepsilon)^{\frac{3\ell}{2}}}{4 d^{\frac{3}{2}}\, \sqrt{|S|}\, (d-1)^{\frac{\ell}{2}}}
\geqslant \frac{d}{\sqrt{n}\, (d-1)}\, d^\ell.
\end{equation}
By Claim \ref{cl3-part-b} and \eqref{eq::01}--\eqref{eq::04}, we conclude that
\[ \bigg\| \sum_{k=1}^{\ell} A^k x\bigg\|_2^2 \geqslant
\frac{L^{3}\, (d-2)\, (d-1-\varepsilon)^{3\ell}}{16\, d^{3}\, (d-1)^{\ell}}. \]
But since $\|x\|_2=1$ and $x_1+\dots+x_n=0$, by Fact \ref{f2-part-b}, we must have that
\[ \frac{L^{3}\, (d-2)\, (d-1-\varepsilon)^{3\ell}}{16\, d^{3}\, (d-1)^{\ell}} \leqslant 4 \big( 4.41(d-1)\big)^\ell. \]
This is easily seen to yield a contradiction by the choice of $\alpha,\varepsilon$ and $L$. This completes
the proof of part (b) of Proposition \ref{lemma:long-range-expansion}.

\section{Remarks and open problems}

%\subsection{From Banach spaces with an unconditional basis to Banach lattices}
%
%The non-linear embedding theorem of Odell--Schlumprecht \cite[Theorem 2.1]{OS94}---communicated here as Theorem~\ref{OS-uniform-homeomorphism} in Appendix \ref{appendix}---was extended by Chaatit \cite{Ch95} to Banach lattices with non-trivial cotype. That is, one can replace the assumption of having an unconditional basis with the assumption of lattice structure. Consequently, the transfer techniques of Ozawa \cite{Oz04} and Naor \cite{Na14} yield a nonlinear Poincar\'{e} inequality when the target space is a Banach lattice with non-trivial cotype. It is natural to expect that our main result can also be extended to this class of spaces.
%\begin{problem} \label{pr-1}
%Extend Theorem \ref{aekjfnakfn} to Banach lattices with non-trivial cotype.
%\end{problem}

\subsection{Optimizing the exponent of the cotype in the estimate for the Poincar\'{e} constant}

The exponent $10$ for the dependence of the Poincar\'{e} constant on the cotype $q$ in our estimates~\eqref{eq-Gamma}
and \eqref{eq-new1} is neither optimal, nor even optimized with respect to the current proof. It would be interesting
as a starting point to determine the limit of our current argument by extracting the best possible exponent from
the proof of Theorem \ref{aekjfnakfn}. More generally, it would be valuable to characterize the optimal dependence
of the Poincar\'e constant on the cotype.
\begin{problem} \label{pr-2}
Is it true that random regular graphs satisfy a discrete Poincar\'{e} inequality with respect to Banach
spaces with an unconditional basis and cotype $q\geqslant 2$, with a Poincar\'{e} constant that depends
\emph{linearly} on the cotype?
\end{problem}

\subsection{Graphs with the long-range expansion property and degree $d\in \{3,4,5\}$}

Proposition~\ref{lemma:long-range-expansion} shows that for any integer $d\geqslant 6$, a uniformly random $d$-regular graph
satisfies part (\hyperref[Part-B]{$B$}) of the long-range expansion property with high probability.
A natural problem is to determine if this also holds true for regular graphs with degree $d\in \{3,4,5\}$.
(Recall that part (\hyperref[Part-A]{$A$}) is satisfied for any integer $d\geqslant 3$.)
\begin{problem} \label{pr-3}
Is it true that for any integer $d\in\{3,4,5\}$, a uniformly random $d$-regular graph satisfies part~\emph{(\hyperref[Part-B]{$B$})}
of the long-range expansion property with high probability?
\end{problem}

\subsection{Explicit constructions of regular graphs with the long-range expansion property}
We recall a problem discussed in the introduction.
\begin{problem} \label{pr-4}
Find explicit examples of regular graphs with the long-range expansion property.
\end{problem}
This problem is closely related to the longstanding open problem of explicitly constructing lossless expanders; see Remark \ref{rem-explicit-constructions} for further discussion.

\appendix

\section{Nonlinear spectral gaps via uniform homeomorphisms} \label{appendix}

\begin{definition}[Concavity] \label{q-concave}
Let $q>1$, and let $M\geqslant 1$. We say that a Banach space $X$ with a $1$-unconditional basis is
\emph{$q$-concave with constant $M$} if for every choice $x_1,\dots,x_m$ of vectors in $X$,
\begin{equation} \label{ea.1}
\Bigg\| \Big( \sum_{i=1}^m |x_i|^q \Big)^{1/q} \Bigg\| \geqslant \frac{1}{M} \Big( \sum_{i=1}^m  \|x_i\|^q\Big)^{1/q}.
\end{equation}
The smallest positive constant $M$ for which \eqref{ea.1} is satisfied is denoted by $M_{(q)}(X)$.
\end{definition}
We shall shortly discuss in more detail the relation between the property of having cotype $q$ and that of being $q$-concave for Banach spaces with a $1$-unconditional basis. At this point we mention, however, that if a Banach space $X$ with a $1$-unconditional basis is $q$-concave ($q\geqslant 2$) with $M_{(q)}(X)=1$, then $X$ has cotype $q$; see \cite[Theorems 1.e.16 and 1.f.1]{LT79}.

We have the following theorem.
\begin{theorem} \label{OS-spectral-gap}
Let $d\geqslant 3$ be an integer, let $G$ be a $d$-regular graph, and let $\lambda_2(G)$ denote the second largest
eigenvalue of the adjacency matrix of $G$.
Also let $X$ be a Banach space with a $1$-unconditional basis, and let $q\geqslant 2$.
\begin{enumerate}
\item[(i)] If $X$ has cotype $q$ with constant $\mathrm{C}\geqslant 1$, then
\[ \gamma(G,\|\cdot\|_X^2) \leqslant 2^{3616q+450}\, q^{384q+104}\, \mathrm{C}^{129q+4}\, \Big( \frac{d}{d-\lambda_2(G)}\Big)^8. \]
\item[(ii)] If, in addition, the space $X$ is $q$-concave with $M_{(q)}(X)=1$, then
\[ \gamma(G,\|\cdot\|_X^2) \leqslant q^{64}\, 2^{576q+234}\, \Big( \frac{d}{d-\lambda_2(G)}\Big)^8. \]
\end{enumerate}
\end{theorem}
Theorem \ref{OS-spectral-gap} follows by combining a comparison result of nonlinear spectral gaps
due to Naor \cite[Theorem 1.10]{Na14} together with the following quantitative version of \cite[Theorem 2.1]{OS94}.
\begin{theorem}[Odell--Schlumprecht] \label{OS-uniform-homeomorphism}
Let $q\geqslant 2$, let $X$ be a finite-dimensional normed space with a $1$-unconditional basis,
and set $n:=\mathrm{dim}(X)$.
\begin{enumerate}
\item[(i)] If $X$ has cotype $q$ with constant $\mathrm{C}\!\geqslant\!1$, then there exists a homeomorphism
$\Phi\colon B(X)\to B(\ell_2^n)$ such that for every $x_1,x_2\in B(X)$ we have
\begin{equation} \label{ea.4}
\!\!\!\frac{\|x_1-x_2\|_X^{24q}}{(1-e^{-1})^{-24q}\, 2^{190q+20}\, q^{24q+6}\, \mathrm{C}^{12q}} \leqslant
\|\Phi(x_1) -\Phi(x_2)\|_{\ell_2^n} \leqslant \frac{8\sqrt{5}}{\sqrt{1-e^{-1}}} \sqrt{q} \sqrt[4]{\mathrm{C}}\, \|x_1-x_2\|_X^{1/8}.
\end{equation}
\item[(ii)] If, in addition, the space $X$ is $q$-concave with $M_{(q)}(X)=1$, then there exists a homeomorphism
$\Psi\colon B(X)\to B(\ell_2^n)$ such that for every $x_1,x_2\in B(X)$ we have
\begin{equation} \label{ea.5}
\frac{\|x_1-x_2\|_X^{12q}}{q^4\, 2^{36q+8}} \leqslant \|\Psi(x_1) -\Psi(x_2)\|_{\ell_2^n} \leqslant
4\sqrt{5}\sqrt[4]{2}\, \|x_1-x_2\|_X^{1/8}.
\end{equation}
\end{enumerate}
\end{theorem}
Since Theorem \ref{OS-uniform-homeomorphism} is not explicitly isolated in the literature, we shall describe its main
steps focusing, in particular, on the estimates \eqref{ea.4} and \eqref{ea.5}.
(See, also, \cite{ANNRW18b} for quantitative versions of closely related results of Daher \cite{Da95} and Kalton.)
We also note that parts of the argument below follow the
more recent exposition of \cite[Theorem 2.1]{OS94} in \cite[Chapter 9]{BL00}.

We start with a couple of auxiliary lemmas that relate the cotype of a Banach space with other basic geometric invariants.
We recall that if $X$ is a Banach space with a $1$-unconditional basis and $p>1$, then $M^{(p)}(X)$ denotes the $p$-convexity
constant of $X$; see \cite[Definition 1.d.3]{LT79}.
\begin{lemma} \label{la.4}
Let $q\geqslant 2$, and let $X$ be a Banach space with a $1$-unconditional basis that has cotype $q$ with constant $\mathrm{C}\geqslant 1$.
Then, for every $s>q$, the space $X$ is $s$-concave with
\[ M_{(s)}(X) \leqslant (1-e^{-1})^{-2} \, s\, \Bigg( \frac{2q(s-1)}{s-q}\Bigg)^{1-\frac{1}{s}} \, \mathrm{C}. \]
\end{lemma}
\begin{proof}
We first observe that, since $X$ has a $1$-unconditional basis and cotype $q$ with constant $\mathrm{C}$, the space $X$ satisfies
property ($\substack{\ast \\ \ast}$) in \cite[Definition 1.f.4]{LT79} with the same $q$ and constant $M=\mathrm{C}$. Thus,
setting $p:=\frac{q}{q-1}$ to be the conjugate exponent of $q$, by \cite[Proposition 1.f.5]{LT79}, the dual
$X^*$ of $X$ satisfies property ($\ast$) in \cite[Definition 1.f.4]{LT79} with $p$ and constant $M=\mathrm{C}$.
Fix any $s>q$, let $s'$ denote the conjugate exponent of $s$, and notice that $1<s'<p$. By the previous
discussion, \cite[Theorem 1.f.7]{LT79} and after keeping track of the constants in \cite[Lemma 1.f.8]{LT79},
we obtain that the space $X^*$ is $s'$-convex with
$M^{(s')}(X^*) \leqslant (1-e^{-1})^{-2} \, \frac{s'}{s'-1}\, \big( \frac{2p}{p-s'}\big)^{1/s'} \, \mathrm{C}$.
The proof is completed by combining the previous estimate with \cite[Proposition 1.d.4]{LT79}.
\end{proof}
For the next lemma we recall that, for any Banach space $X$, $\delta_X(\varepsilon)$ denotes the modulus of convexity
of $X$ at $0<\varepsilon \leqslant 2$, while $\rho_X(\tau)$ denotes the the modulus of uniform smoothness of $X$ at $\tau>0$;
see \cite[Definition 1.e.1]{LT79}.
\begin{lemma}\label{la.5}
Let $q\geqslant 2$, and let $Y$ be a Banach space with a $1$-unconditional basis such that $M^{(2)}(Y)=M_{(q)}(Y)=1$.
Then, for every $0<\varepsilon \leqslant 2$ and every $\tau>0$, we have
\begin{equation} \label{ea.7}
\delta_Y(\varepsilon) \geqslant \frac{\varepsilon^q}{q\, 4^q} \ \ \ \text{ and } \ \ \
\rho_Y(\tau) \leqslant 2\tau^2.
\end{equation}
\end{lemma}
\begin{proof}
The estimate on $\delta_Y(\varepsilon)$ in \eqref{ea.7} follows from the assumption that $M_{(q)}(Y)=1$
and \cite[Theorem 1.f.1]{LT79}, after noticing that in \cite[Lemma 1.f.2]{LT79} we can take $C(2,q)=4$.
For the upper bound on $\rho_Y(\tau)$ we observe that, by \cite[Proposition 1.d.4]{LT79},
we have $M_{(2)}(Y^*)=M^{(2)}(Y)=1$. Thus, by the previous discussion, we obtain that
$\delta_{Y^*}(\varepsilon)\geqslant \frac{\varepsilon^2}{32}$ for every $0<\varepsilon \leqslant 2$.
Finally, by \cite[Proposition 1.e.2]{LT79}, we have
$\rho_Y(\tau)=\sup\{\tau\varepsilon/2-\delta_{Y^*}(\varepsilon)\colon 0\leqslant \varepsilon\leqslant 2\}$ for any $\tau>0$,
which is easily seen to imply that $\rho_Y(\tau) \leqslant 2\tau^2$.
\end{proof}
After these preliminary results, we are ready to begin the sketch of the proof of Theorem \ref{OS-uniform-homeomorphism}.
We shall only discuss the proof of part (i); the proof of part (ii) is actually simpler since Step 2 below is not needed.
So, in what follows, we fix $q\geqslant 2$ and an $n$-dimensional normed space $X$ with a $1$-unconditional basis that has
cotype $q$ with constant $\mathrm{C}\geqslant 1$.
\medskip

\noindent \textbf{Step 1: passing to the $2$-convexification.} Let $X^{(2)}$ denote the $2$-convexification of $X$; we refer to
\cite[page 53]{LT79} for the definition of this construction, and we recall at this point that~$X^{(2)}$ also has
a $1$-unconditional basis. By \cite[Proposition 9.3]{BL00} (and after keeping track of the constants in its proof),
there exists a homeomorphism $F_1\colon S(X)\to S(X^{(2)})$ such that for every $x_1,x_2\in S(X)$ we have
$\frac{1}{2}\|x_1-x_2\|_X\leqslant \|F_1(x_1)-F_1(x_2)\|_{X^{(2)}} \leqslant \sqrt{2} \|x_1-x_2\|_X^{1/2}$.
\medskip

\noindent \textbf{Step 2: renorming of $X^{(2)}$.} By \cite[pages 53 and 54]{LT79}, we have
$M^{(2)}(X^{(2)})=1$ and $M_{(2s)}(X^{(2)})\leqslant \sqrt{M_{(s)}(X)}$ for every $s>q$; consequently,
by Lemma \ref{la.4},
\[ M:=M_{(4q)}(X^{(2)})\leqslant \sqrt{M_{(2q)}(X)}\leqslant \frac{2\sqrt{2}}{1-e^{-1}}\, q\, \sqrt{\mathrm{C}}. \]
(In particular, if $M_{(q)}(X)=1$, then $M_{(2q)}(X^{(2)})=1$.)
By \cite[pages 184 and 185]{FJ74}, there exists an equivalent norm $\seminorm{\cdot}$ on $X^{(2)}$
such that the space $Y:=(X^{(2)},\seminorm{\cdot})$ also has a $1$-unconditional basis
and it satisfies $M^{(2)}(Y)=M_{(4q)}(Y)=1$. Moreover, we have
$\|x\|_{X^{(2)}} \leqslant \seminorm{x} \leqslant M \|x\|_{X^{(2)}}$ for every $x\in X^{(2)}$ and, consequently,
there exists a homeomorphism $F_2\colon S(X^{(2)})\to S(Y)$ such that
$\frac{1}{2M}\|x_1-x_2\|_{X^{(2)}} \leqslant \seminorm{F_2(x_1)-F_2(x_2)} \leqslant 2M \|x_1-x_2\|_{X^{(2)}}$
for every $x_1,x_2\in S(X^{(2)})$.
\medskip

\noindent \textbf{Step 3: homeomorphism of spheres.} Since $M^{(2)}(Y)=M_{(4q)}(Y)=1$, by Lemma \ref{la.5}, we have that
$\delta_Y(\varepsilon) \geqslant \frac{\varepsilon^{4q}}{q\, 4^{4q+1}}$ for every $0< \varepsilon \leqslant 2$ and
$\rho_Y(\tau) \leqslant 2\tau^2$ for every $\tau>0$. On the other hand, by \cite[Lemma 9.5 and Proposition A.5]{BL00},
there exists a homeomorphism $F_3\colon S(Y)\to S(\ell^n_1)$ such that
$\frac{1}{8} \delta_Y\big(\seminorm{y_1-y_2}\big)^3 \leqslant \|F_3(y_1)-F_3(y_2)\|_{\ell^n_1} \leqslant
\frac{\rho_Y(2\seminorm{y_1-y_2})}{\seminorm{y_1-y_2}} + \seminorm{y_1-y_2}$ for every $y_1,y_2\in S(Y)$.
Next observe that the $2$-convexification $(\ell^n_1)^{(2)}$ of $\ell^n_1$ is isometric to $\ell^n_2$; thus, applying
\cite[Proposition~9.3]{BL00} as in Step 1, we obtain a homeomorphism $F_4\colon S(\ell^n_1)\to S(\ell^n_2)$ such that
for every $z_1,z_2\in S(\ell^n_1)$ we have
$\frac{1}{2}\|z_1-z_2\|_{\ell^n_1}\leqslant \|F_4(z_1)-F_4(z_2)\|_{\ell^n_2} \leqslant \sqrt{2} \|z_1-z_2\|_{\ell^n_1}^{1/2}$.

Summing up, setting $\phi:=F_4\circ F_3 \circ F_2\circ F_1$,
we see that $\phi\colon S(X)\to S(\ell^n_2)$ is a homeomorphism that satisfies
$\frac{\|x_1-x_2\|_X^{12q}}{q^3\, 2^{60q+9}\, M^{12q}} \leqslant
\|\phi(x_1) -\phi(x_2)\|_{\ell_2^n} \leqslant 2\sqrt{5}\sqrt[4]{2} \sqrt{M} \|x_1-x_2\|_X^{1/4}$
for every $x_1,x_2\in S(X)$.\!
\medskip

\noindent \textbf{Step 4: homeomorphism of balls.} By \cite[Proposition 2.9]{OS94} (and its proof), we have that
if $f\colon S(E)\to S(Z)$ is a homeomorphism between the spheres of two Banach spaces $E$ and $Z$
that satisfies $\alpha\big(\|e-w\|_E\big)\leqslant \|f(e)-f(w)\|_Z \leqslant \beta\big(\|e-w\|_E\big)$ for every $e,w\in S(E)$,
then there exists a homeomorphism $F\colon B(E)\to B(Z)$ such that for every $e_1,e_2\in B(E)$,
\begin{enumerate}
\item[$\bullet$] $\|F(e_1)-F(e_2)\|_Z \leqslant \|e_1-e_2\|_E + \max\Big\{ 2\|e_1-e_2\|_E^{1/4},
\beta\big( 2\|e_1-e_2\|^{1/2}_E\big)\Big\}$, and
\item[$\bullet$] $\|e_1-e_2\|_E \leqslant \|F(e_1)-F(e_2)\|_Z + \max\Big\{ 2\|F(e_1)-F(e_2)\|_Z^{1/4},
\alpha^{-1}\big( 2\|F(e_1)-F(e_2)\|_Z^{1/2}\big)\Big\}$.
\end{enumerate}
The proof of Theorem \ref{OS-uniform-homeomorphism} is then completed
by applying this result to the homeomorphism $\phi$ obtained in Step 3.

\subsection*{Acknowledgments}

The research was supported in the framework of H.F.R.I. call ``Basic research Financing
(Horizontal support of all Sciences)" under the National Recovery and Resilience Plan ``Greece 2.0"
funded by the European Union--NextGenerationEU (H.F.R.I. Project Number: 15866).
The third named author (K.T.) is partially supported by the NSF grant DMS 2331037.


\begin{thebibliography}{99}

\bibitem[ANNRW18a]{ANNRW18a} A. Andoni, A. Naor, A. Nikolov, I. Razenshteyn and E. Waingarten,
\emph{Data-dependent hashing via nonlinear spectral gaps}, in ``Proceedings of the 50th Annual
ACM SIGACT Symposium on Theory of Computing"\!, STOC 18, 787--800.

\bibitem[ANNRW18b]{ANNRW18b} A. Andoni, A. Naor, A. Nikolov, I. Razenshteyn and E. Waingarten,
\emph{H\"{o}lder homeomorphisms and approximate nearest neighbors}, in ``59th IEEE Annual Symposium on
Foundations of Computer Science"\!, FOCS 2018, 159--169.

\bibitem[BL00]{BL00} Y. Benyamini and J. Lindenstrauss,
\emph{Geometric Nonlinear Functional Analysis. Volume 1},
Colloquium Publications, Vol. 48, American Mathematical Society, 2000.

\bibitem[Bol80]{Bol80} B. Bollob\'{a}s,
\emph{A probabilistic proof of an asymptotic formula for the number of labelled regular graphs},
European J. Combin. 1 (1980), 311--316.

\bibitem[Bol01]{Bol01} B. Bollob\'{a}s,
\emph{Random Graphs},
Cambridge Studies in Advanced Mathematics, vol. 73, Cambridge University Press, 2001.

\bibitem[CRVW02]{CRVW02} M. Capalbo, O. Reingold, S. Vadhan and A. Wigderson,
\emph{Randomness conductors and constant-degree lossless expanders},
in ``Proceedings of the 34th Annual ACM Symposium on Theory of Computing"\!, STOC 02, 659--668.

%\bibitem[Ch95]{Ch95} F. Chaatit,
%\emph{On uniform homeomorphisms of the unit spheres of certain Banach lattices},
%Pacific J. Math. 168 (1995), 11--31.

\bibitem[Da95]{Da95} M. Daher,
\emph{Hom\'{e}omorphismes uniformes entre les sph\`{e}res unit\'{e} des espaces d’interpolation},
Canad. Math. Bull. 38 (1995), 286--294.

\bibitem[Es22]{Es22} A. Eskenazis,
\emph{Average distortion embeddings, nonlinear spectral gaps, and a metric John theorem $[$after Assaf Naor$]$},
S\'{e}minaire Bourbaki, Volume 2021/2022, Ast\'{e}risque 438 (2022), 295--333.

\bibitem[FJ74]{FJ74} T. Figiel and W. B. Johnson,
\emph{A uniformly convex Banach space which contains no $\ell_p$},
Compos. Math. 29 (1974), 179--190.

\bibitem[Fr08]{Fr08} J. Friedman,
\emph{A proof of Alon’s second eigenvalue conjecture and related problems},
Memoirs Amer. Math. Soc. 910 (2008), 1--100.

\bibitem[Gr03]{Gr03} M. Gromov,
\emph{Random walk in random groups},
Geom. Funct. Anal. 13 (2003), 73--146.

\bibitem[HLW06]{HLW06} S. Hoory, N. Linial and A. Wigderson,
\emph{Expander graphs and their applications},
Bull. Amer. Math. Soc. 43 (2006), 439--561.

\bibitem[JLR00]{JLR00} S. Janson, T. {\L}uczak, and A. Ruci\'{n}ski,
\emph{Random Graphs},
Wiley-Interscience Series in Discrete Mathematics and Optimization, Wiley-Interscience, 2000.

\bibitem[JL01]{JL01} W. B. Johnson and J. Lindenstrauss,
\emph{Basic concepts in the geometry of Banach spaces},
in ``Handbook of the Geometry of Banach Spaces. Volume 1"\!, Elsevier, 2001, 1--84.

\bibitem[Ka95]{Ka95} N. Kahale,
\emph{Eigenvalues and expansion of regular graphs},
Journal of the ACM 42 (1995), 1091--1106.

\bibitem[dLdS23]{dLdS23}
T. de Laat and M. de la Salle,
\emph{Actions of higher rank groups on uniformly convex Banach spaces},
preprint (2023), available at \url{https://arxiv.org/abs/2303.01405}.

\bibitem[La08]{La08} V. Lafforgue,
\emph{Un renforcement de la propri\'{e}t\'{e} $(T)$},
Duke Math. J. 143 (2008), 559--602.

\bibitem[LT79]{LT79} J. Lindenstrauss and L. Tzafriri,
\emph{Classical Banach spaces. II: Function spaces},
Ergebnisse der Mathematik und ihrer Grenzgebiete, Vol. 97, Springer-Verlag, 1979.

\bibitem[LS10]{LS10} E. Lubetzky and A. Sly,
\emph{Cutoff phenomena for random walks on random regular graphs},
Duke Math. J. 153 (2010), 475--510.

\bibitem[LPS88]{LPS88} A. Lubotzky, R. Phillips and P. Sarnak,
\emph{Ramanujan graphs},
Combinatorica 8 (1988), 261--277.

\bibitem[Ma97]{Ma97} J. Matou\v{s}ek,
\emph{On embedding expanders into $\ell_p$ spaces},
Israel J. Math. 102 (1997), 189--197.

\bibitem[Ma11]{Ma11}  J. Matou\v{s}ek,
\emph{Open problems on embeddings of finite metric spaces},
available at \url{https://kam.mff.cuni.cz/~matousek/metrop.ps}.

\bibitem[Mau74]{Mau74} B. Maurey,
\emph{Type et cotype dans les espaces munis de structures locales inconditionnelles},
S\'{e}min. Maurey--Schwartz 1973-1974, Espaces $L^p$, Appl. radonif., G\'{e}om. Espaces de Banach, Expos\'{e}s No. 24 et 25, 25 p. (1974).

\bibitem[MP76]{MP76} B. Maurey and G. Pisier,
\emph{S\'{e}ries de variables al\'{e}atoires vectorielles ind\'{e}pendantes et
propri\'{e}t\'{e}s g\'{e}om\'{e}otriques des espaces de Banach},
Studia Math. 58 (1976), 45--90.

\bibitem[MS21]{MS21} T. McKenzie and S. Mohanty,
\emph{High-girth near-{R}amanujan graphs with lossy vertex expansion},
in ``International Colloquium on Automata, Languages and Programming"\!, ICALP 2021, 1--15.

\bibitem[MN14]{MN14} M. Mendel and A. Naor,
\emph{Nonlinear spectral calculus and super-expanders},
Publ. Math. Inst. Hautes \'{E}tud. Sci. 119 (2014), 1--95.

\bibitem[MN15]{MN15} M. Mendel and A. Naor,
\emph{Expanders with respect to Hadamard spaces and random graphs},
Duke Math. J. 164 (2015), 1471--1548.

\bibitem[Na14]{Na14} A. Naor,
\emph{Comparison of metric spectral gaps},
Anal. Geom. Metr. Spaces 2 (2014), 1--52.

\bibitem[NS11]{NS11} A. Naor and L. Silberman,
\emph{Poincar\'{e} inequalities, embeddings, and wild groups},
Compos. Math. 147 (2011), 1546--1572.

\bibitem[OS94]{OS94} E. Odell and Th. Schlumprecht,
\emph{The distortion problem},
Acta Math. 173 (1994), 259--281.

\bibitem[Opp23]{Opp23} I. Oppenheim,
\emph{Banach property $(T)$ for $\mathrm{SL}_n(\mathbb{Z})$ and its applications},
Invent. Math. 234 (2023), 893--930.

\bibitem[Oz04]{Oz04} N. Ozawa,
\emph{A note on non-amenability of $\mathcal{B}(\ell_p)$ for $p=$1,\,2},
Internat. J. Math. 15 (2004), 557--565.

\bibitem[Pi10]{Pi10} G. Pisier,
\emph{Complex interpolation between Hilbert, Banach and operator spaces},
Memoirs Amer. Math. Soc. 978 (2010), 1--78.

\bibitem[Vad12]{Vad12} S. P. Vadhan,
\emph{Pseudorandomness},
in ``Foundations and Trends$\circledR$ in Theoretical Computer Science"\!, 2012, 1--336.

\bibitem[Wo99]{Wo99}  N. C. Wormald,
\emph{Models of random regular graphs},
in ``Surveys in Combinatorics, 1999"\!, LMS Lecture Note Series, vol. 267, Cambridge University Press, 1999, 239--298.

\end{thebibliography}
\end{document}